\documentclass[journal,11pt,draftcls,onecolumn]{IEEEtran}
\usepackage{cite}
\usepackage{graphicx}
\usepackage{psfrag}
\usepackage{epsfig}
\usepackage{url}
\usepackage{stfloats}
\usepackage{amsfonts,amssymb,amsmath,bm,paralist,theorem,cite,ifthen,color,amsmath}
\usepackage{array}
\usepackage{caption}
\usepackage{calc}
\usepackage{longtable}
\usepackage{subfigure}
\usepackage{mathrsfs}
\usepackage{epsfig}
\newtheorem{prop}{Proposition}
\newtheorem{thm}{Theorem}

\newtheorem{coro}{Corollary}

\theorembodyfont{\rmfamily}

\newcommand{\tr}{{\rm Tr}}
\newcommand{\st}{{\rm s.t.}}
\newcommand{\re}{\mathbb{R}}

\newcommand{\bv}{{\textbf{v}}}
\newcommand{\bV}{{\textbf{V}}}
\newcommand{\bH}{{\mathbf{H}}}

\newcommand{\bs}{\mathbf{s}}

\newcommand{\bx}{\mathbf{x}}
\newcommand{\by}{\mathbf{y}}
\newcommand{\bI}{\mathbf{I}}
\newcommand{\bn}{\mathbf{n}}
\newcommand{\bE}{\mathbf{E}}
\newcommand{\bU}{\mathbf{U}}
\newcommand{\bu}{\mathbf{u}}

\newtheorem{assumption}{Assumption}

\begin{document}
\title{A Unified Convergence Analysis of Block Successive Minimization Methods for Nonsmooth Optimization}

%
% Single address.
% ---------------

\author{Meisam Razaviyayn, Mingyi Hong and Zhi-Quan Luo%~\IEEEmembership{Fellow,~IEEE}% <-this % stops a space
\thanks{The authors are with the Department of Electrical and Computer Engineering, University of Minnesota, 200 Union
Street SE, Minneapolis, MN 55455. Emails: \{meisam,mhong,luozq\}@ece.umn.edu.}}

%\MARKBOTH{ieee tRANSACTIONS ON iNFORMATION tHEORY~(SUBMITTED)}%
%{ieee tRANSACTIONS ON iNFORMATION tHEORY~(SUBMITTED)}
%\ninept
%
\maketitle
%\date{March 28, 2010}
%
\begin{abstract}
The block coordinate descent (BCD) method is widely used for
minimizing a continuous function $f$ of several block variables. At each
iteration of this method, a single block of variables is optimized,
while the remaining variables are held fixed. To ensure the
convergence of the BCD method, the subproblem to be optimized in
each iteration needs to be solved {\it exactly} to its {\it unique}
optimal solution. Unfortunately, these requirements are often too
restrictive for many practical scenarios. In this paper, we study an
alternative inexact BCD approach which updates the variable blocks
by successively minimizing a sequence of approximations of $f$ which are
either locally tight upper bounds of $f$ or strictly convex local approximations
of $f$. We focus on characterizing the
convergence properties for a fairly wide class of such methods,
especially for the cases where the objective functions are either
non-differentiable or nonconvex. Our results unify and extend the
existing convergence results for many classical algorithms such as the BCD method, the
difference of convex functions (DC) method, the expectation maximization
(EM) algorithm, as well as the alternating proximal minimization algorithm.
\end{abstract}


%\begin{AMS}
%15A15, 15A09, 15A23
%\end{AMS}

%\pagestyle{myheadings}
%\thispagestyle{plain}
%\markboth{TEX PRODUCTION AND V. A. U. THORS}{SIAM MACRO EXAMPLES}
\newpage

\section{Introduction}
\label{sec:intro}
Consider the following optimization problem
\begin{align}
\min \quad &f(x_1,\ldots,x_n)\nonumber\\
\st \quad &x_i \in \mathcal{X}_i, \; i=1,2,\ldots, n,\nonumber
\end{align}
where $\mathcal{X}_i \subseteq  \re^{m_i}$ is a closed convex set,
and $f: \prod_{i=1}^{n}\mathcal{X}_i\to\mathbb{R}$ is a continuous
function. A popular approach for solving the above optimization
problem is the block coordinate descent method (BCD), which is also
known as the Gauss-Seidel method. At each iteration of this method,
the function is minimized with respect to a single block of
variables while the rest of the blocks are held fixed. More
specifically, at iteration $r$ of the algorithm, the block variable
$x_i$ is updated by solving the following subproblem
\begin{align}
x_i^r = \arg \min_{y_i \in \mathcal{X}_i} \quad
f(x_1^r,\ldots,x_{i-1}^r,y_i,x_{i+1}^{r-1},\ldots,x_n^{r-1}),\quad
i=1,2,\ldots,n \label{EQ:BCD}.
\end{align}
 Let us use $\{x^r\}$ to denote
the sequence of iterates generated by this algorithm, where $x^r
\triangleq (x_1^r,\ldots,x_n^r)$.  Due to its particular simple
implementation, the BCD method has been widely used for solving
problems such as power allocation in wireless communication systems
\cite{WMMSETSP}, clustering \cite{HartiganKmeans}, image denoising
and image reconstruction \cite{Imagereconstruction} and
dynamic programming \cite{DynamicProgramming}. \\

Convergence of the BCD method typically requires the uniqueness of
the minimizer at each step or the quasi-convexity of the objective
function (see \cite{BCDTseng} and the references therein). Without
these assumptions, it is possible that the BCD iterates do not get
close to any of the stationary points of the problem (see Powell
\cite{PowellBCD73} for examples). Unfortunately, these requirements
can be quite restrictive in some important practical
problems such the tensor decomposition problem (see
\cite{TensorReviewKolda} and the application section in this work)
and the sum rate maximization problem in wireless networks. In fact,
for the latter case, even solving the per block
subproblem~\eqref{EQ:BCD} is difficult due to the non-convexity and
non-differentiability of the objective function.

To overcome such difficulties, one can modify the BCD algorithm by
optimizing a well-chosen {\it approximate} version of the objective
function at each iteration. The classical gradient descent method,
for example, can be viewed as an implementation of such strategy. To
illustrate, recall that the update rule of the gradient descent
method is given by
\[
x^{r+1} = x^r - \alpha^{r+1} \nabla f(x^r).
\]
This update rule is equivalent to solving the following problem
\[
x^{r+1} = \arg \min_{x} \quad g(x,x^r),
\]
where
\[
g(x,x^r) \triangleq f(x^r) + \nabla f(x^r)(x-x^r) +
\frac{1}{2\alpha^{r+1}} \|x-x^r\|^2.
\]
Clearly, the function $g(x,x^r)$ is an approximation of $f(\cdot)$
around the point $x^r$. In fact, as we will see later in this paper,
successively optimizing an approximate version of the original
objective is the key idea of many important algorithms such as the
concave-convex procedure \cite{CCCP}, the EM algorithm
\cite{EMDempster}, the proximal minimization algorithm
\cite{Bertsekas_Book_Nonlinear}, to name a few. Furthermore, this
idea can be used to simplify the computation and to guarantee the
convergence of the original BCD algorithm with the Gauss-Seidel
update rule (e.g. \cite{CGDTseng}, \cite{ProximalBCD},\cite{WataoYinBCD}). However,
despite its wide applicability, there appears to be no general unifying
convergence analysis for this class of algorithms.

In this paper, we provide a unified convergence analysis for a general class
of inexact BCD methods in which a sequence of approximate versions of
the original problem are solved successively. Our focus will be on
problems with nonsmooth and nonconvex objective functions. Two types of approximations are considered: one being a locally tight upper bound for the original objective function, the other being a convex local approximation of the objective function. We provide convergence analysis for both of these successive approximation strategies as well as for various types of updating rules, including the cyclic updating rule, the Gauss-Southwell update rule or the overlapping essentially cyclic update rule. By allowing inexact solution of subproblems,
our work unifies and extends several existing algorithms and their convergence analysis, including the difference of convex functions (DC)
method, the expectation maximization (EM) algorithm, as well as the
alternating proximal minimization algorithm.

%In this paper, we combine the idea of convex approximation and block coordinate descent approach to propose different iterative algorithms. The convergence to a stationary point of the problem is established under different sets of assumptions.

\section{Technical Preliminaries}
Throughout the paper, we adopt the following notations. We use
$\mathbb{R}^m$ to denote the space of~$m$ dimensional real valued
vectors, which is also represented as the Cartesian product
of~$n$ smaller real valued vector spaces, i.e.,
\[
\re^{m} = \re^{m_1} \times \re^{m_2} \times \ldots \times \re^{m_n},
\]
where $\sum_{i=1}^{n} m_i = m$. We use the notation
$(0,\ldots,d_k,\ldots,0)$ to denote the vector of all zeros except
the $k$-th block, with $d_k \in \re^{m_k}$. The following
concepts/definitions are adopted in our paper:
\begin{itemize}
\item \textbf{Distance of a point from a set:} Let~$\mathcal{S} \subseteq \re^m$ be a set and $x$ be a point in~$\mathbb{R}^m$, the distance of the point~$x$ from the set~$\mathcal{S}$ is defined as
\[
d(x,\mathcal{S}) = \inf_{s \in \mathcal{S}} \|x-s\|,
\]
where~$\|\cdot\|$ denotes the 2-norm in~$\mathbb{R}^m$.

\item \textbf{Directional derivative:} Let $f: \mathcal{D} \rightarrow \re $ be a function where~$\mathcal{D} \subseteq \re^m$ is a convex set. The directional derivative of~$f$ at point~$x$ in direction~$d$ is defined by
\[
f'(x;d) \triangleq \liminf_{\lambda \downarrow 0} \frac{f(x+\lambda d) - f(x)}{ \lambda}.
\]

\item \textbf{Stationary points of a function:} Let $f: \mathcal{D} \rightarrow \re $ be a function
where~$\mathcal{D} \subseteq \re^m$ is a convex set.  The point~$x$
is a stationary point of~$f(\cdot)$ if $f'(x;d) \geq 0$ for all $d$
such that $x + d \in \mathcal{D}$. In this paper we use the notation
$\mathcal{X}^*$ to denote the set of stationary points of a
function.

\item \textbf{Regularity of a function at a point:} The function~$f:\re^m \rightarrow \re$ is
regular at the point~$z \in {\rm dom} f$ with respect to the
coordinates~$m_1,m_2,\ldots,m_n$,  $m_1 + m_2 + \ldots +  m_n = m$,
if~$f'(z;d)\geq 0$ for all $d = (d_1,d_2,\ldots, d_n)$ with $
f'(z;d_k^0)\geq 0$, where $d_k^0 = (0\ldots,d_k,\ldots, 0)$ and $d_k
\in \re^{m_k}, \forall\; k$. For detailed discussion on the regularity
of a function, the readers are referred to~\cite[Lemma
3.1]{BCDTseng}.

\item \textbf{Quasi-convex function:} The function~$f$ is quasi-convex if
\[
f(\theta x + (1-\theta)y) \leq \max \{f(x),f(y)\}, \quad \forall\; \theta \in (0,1), \ \forall\; x,y \in \;{\rm dom} \;f
\]

\item \textbf{Coordinatewise minimum of a function:} $z \in {\rm dom} \; f \subseteq \re^{m}$ is coordinatewise minimum of~$f$ with respect to the coordinates in $\Re^{m_1},\Re^{m_2},\ldots, \Re^{m_n}$, $m_1 + \ldots + m_k = m$ if
     \[
     f(z + (0,\ldots,d_k,\ldots, 0)) \geq f(z),\quad \forall\;  d_k \in \re^{m_k} \quad {\rm with} \quad z +  (0,\ldots,d_k,\ldots, 0) \in \;{\rm dom}\; f.
     \]
\end{itemize}

\section{Successive Upper-bound Minimization (SUM)}
\label{sec:SUM}

To gain some insights to the general inexact BCD method, let us
first consider a simple Successive Upper-bound
Minimization (SUM) approach in which all the variables are grouped
into a {\it single} block. Although simple in form, the SUM
algorithm is the key to many important algorithms such as the DC
programming \cite{CCCP} and the EM algorithm \cite{EMTutorial}.

Consider the following optimization problem
\begin{equation}
\label{eq:OriginalProblem}
\begin{split}
\min \quad & f(x) \\
\st \quad & x \in \mathcal{X},
\end{split}
\end{equation}
where $\mathcal{X}$ is a closed convex set. Without loss of generality, we can
 assume that ${\rm dom} \; f = \mathcal{X}$.
When the objective function $f(\cdot)$ is non-convex and/or
nonsmooth, solving \eqref{eq:OriginalProblem} directly may not be
easy. The SUM algorithm circumvents such difficulty by optimizing a
sequence of approximate objective functions instead. More
specifically, starting from a feasible point $x^0$, the algorithm
generates a sequence $\{x^r\}$ according to the following update
rule
\begin{equation}\label{eq:SUMupdaterule}
x^r \in \arg \min_{x \in \mathcal{X}} \quad  u (x,x^{r-1})
\end{equation}
where $x^{r-1}$ is the point generated by the algorithm at
iteration~$r-1$ and $u(x,x^{r-1})$ is an approximation of~$f(x)$ at the
$r$-th iteration. Typically the approximate function
$u(\cdot,\cdot)$ needs to be chosen such that the subproblem
\eqref{eq:SUMupdaterule} is easy to solve. Moreover, to ensure the
convergence of the SUM algorithm, certain regularity conditions on
$u(\cdot,\cdot)$ is required (which will be discussed shortly).
Among others, $u(x,x^{r-1})$ needs to be a global upper bound for
$f(x)$, hence the name of the algorithm. The main steps of the SUM
algorithm are presented in Fig.~\ref{Fig:SUMAlgorithm}.
\begin{figure}[t]
\centering
\begin{tabular}{|p{3.3in}|}
%\hline \textbf{input: $x_0,\alpha_k$, output: $x_k$}\\
\hline
\begin{itemize}
\item [1] \;Find a feasible point $x^0\in \mathcal{X}$ and set $r = 0$
\item [2] \; \textbf{repeat}
\item [3] \quad $r = r+1$
\item [4] \quad  Let $\mathcal{X}^{r} = \arg \min_{x\in \mathcal{X}} u(x , x^{r-1})$
\item [5] \quad Set $x^r$ to be an arbitrary element in~$\mathcal{X}^{r}$
\item [6] \; \textbf{until} some convergence criterion is met
\end{itemize}
\\
\hline
\end{tabular}\vspace{1.2em}
\caption{Pseudo code of the SUM algorithm} \label{Fig:SUMAlgorithm}
\vspace{-0.5cm}
\end{figure}

We remark that the proposed SUM algorithm is in many ways similar to
the inner approximation algorithm (IAA) developed
in~\cite{SCAMarksWright}, with the following key differences:
\begin{itemize}
\item The IAA algorithm approximates {\it both} the objective functions
and the feasible sets. On the contrary, the SUM algorithm only
approximates the objective function.

\item The the IAA algorithm is only applicable for problems with smooth objectives, while the
SUM algorithm is able to handle nonsmooth objectives as well.
\end{itemize}

It is worth mentioning that the existing convergence result for the
IAA algorithm is quite weak. In particular, \cite[Theorem
1]{SCAMarksWright} states that if the whole sequence converges, then
the algorithm should converge to a stationary point. In the
following, we show that the SUM algorithm provides stronger
convergence guarantees as long as the
approximation function $u(\cdot, \cdot)$ satisfies certain mild assumptions\footnote{These assumptions are weaker
than those made to ensure the convergence of the IAA
algorithm. } which we outline below.

\begin{assumption}
\label{AssumptionA}
Let the approximation function $u(\cdot,\cdot)$ satisfy the following
\begin{align}
& u(y, y) = f(y),\quad \forall\; y \in \mathcal{X} \label{A1}\tag{A1} \\
& u(x,y) \geq f(x),\quad \forall\; x,y \in \mathcal{X}\label{A2} \tag{A2}\\
&  u'(x,y;d)\bigg|_{x=y} =  f' (y;d), \quad \forall\; d \;\;{\rm with} \;\; y + d \in \mathcal{X} \label{A3}\tag{A3}\\
& u(x,y) \; {\rm is} \; {\rm continuous} \; {\rm in} \; (x,y) \tag{A4}
\end{align}
\end{assumption}

The assumptions \eqref{A1} and \eqref{A2} imply that the approximate
function $u(\cdot,x^{r-1})$ in \eqref{eq:SUMupdaterule} is a tight
upper bound of the original function. The assumption~\eqref{A3}
guarantees that the first order behavior of~$u(\cdot,x^{r-1})$ is
the same as~$f(\cdot)$ locally (note that the directional derivative
$u'(x,y;d)$ is only with respect to the variable $x$). Although
directly checking \eqref{A3} may not be easy, the following
proposition provides a sufficient condition under which \eqref{A3}
 holds true automatically.

\begin{prop} \label{lemma:SUMAssumption}
Assume $f(x) = f_0(x) + f_1(x)$, where $f_0(\cdot)$ is continuously differentiable and  the directional derivative of $f_1(\cdot)$ exists at every point~$x\in \mathcal{X}$. Consider $u(x,y) = u_0(x,y) + f_1(x)$, where $u_0(x,y)$ is a continuously differentiable function satisfying the following conditions
\begin{align}
& u_0(y , y) = f_0(y),\quad \forall\; y \in \mathcal{X} \label{TightnessCondLemma}\\
& u_0(x,y) \geq f_0(x),\quad \forall\; x,y \in \mathcal{X}.  \label{UpperBoundCondLemma}
\end{align}
Then, \eqref{A1}, \eqref{A2} and \eqref{A3} hold for
$u(\cdot,\cdot)$.
\end{prop}
\begin{proof}
First of all, \eqref{TightnessCondLemma} and \eqref{UpperBoundCondLemma} imply \eqref{A1} and \eqref{A2} immediately. Now we prove \eqref{A3} by contradiction. Assume the contrary so that there exist a $y\in \mathcal{X}$ and a $d \in \re^m$ so that
\begin{equation}\label{ContraryLemma}
 f'(y;d) \neq  u'(x,y;d)\bigg|_{x=y} \quad {\rm and} \quad y+d \in \mathcal{X}.
\end{equation}
This further implies that
\[
f_0'(y;d) \neq  u_0'(x,y;d)\bigg|_{x=y}.
\]
Furthermore, since $f_0(\cdot)$ and $u_0(\cdot,\cdot)$ are continuously differentiable, there exists a $\alpha>0$ such that for $z = y+\alpha d$,
\begin{equation}\label{eq:lemmatemp1}
f'_0 (z; d) \neq u'_0 (x,z; d)\bigg|_{x=z}.
\end{equation}
The assumptions \eqref{TightnessCondLemma} and \eqref{UpperBoundCondLemma} imply that
\begin{equation}\label{eq:lemmatemp2}
\begin{split}
u'_0 (x,z; d)\bigg|_{x=z} &= \lim_{\lambda \downarrow 0} \frac{u_0(z + \lambda d ,z) - u_0(z,z)}{\lambda}\\
& \geq \lim_{\lambda \downarrow 0} \frac{f_0(z + \lambda d) - f_0(z)}{\lambda} = f'(z;d).
\end{split}
\end{equation}
On the other hand, the differentiability of  $f_0(\cdot)$, $u_0(\cdot,\cdot)$ and using \eqref{TightnessCondLemma}, \eqref{UpperBoundCondLemma} imply
\begin{equation}\label{eq:lemmatemp3}
\begin{split}
u'_0 (x,z; d)\bigg|_{x=z} &= \lim_{\lambda \downarrow 0} \frac{u_0(z ,z) - u_0(z- \lambda d ,z)}{\lambda}\\
& \leq \lim_{\lambda \downarrow 0} \frac{f_0(z)  - f_0(z - \lambda d)}{\lambda} = f'(z;d).
\end{split}
\end{equation}
Clearly, \eqref{eq:lemmatemp2} and \eqref{eq:lemmatemp3} imply that $u'_0 (x,z; d)\bigg|_{x=z} = f'(z;d)$ which contradicts  \eqref{eq:lemmatemp1}.
\end{proof}

The following theorem establishes the convergence for the SUM
algorithm.

\begin{thm}\label{thm:ConvergenceSUMConvex}
Assume that Assumption~\ref{AssumptionA} is satisfied. Then every
limit point of the  iterates generated by the SUM algorithm is a
stationary point of the problem \eqref{eq:OriginalProblem}.
\end{thm}
\begin{proof}
Firstly, we observe the following series of inequalities
\begin{equation}\label{eq:decreasingobjval}
f(x^{r+1}) \stackrel{\rm (i)}\leq u(x^{r+1},x^{r}) \stackrel{\rm
(ii)}\leq u(x^r,x^r) = f(x^r),\quad \forall\; r=0,1,2,\ldots
\end{equation}
where step ${\rm (i)}$ is due to \eqref{A1}, step ${\rm (ii)}$ follows from the
optimality of $x^{t+1}$ (cf.\ step 4 and 5 in
Fig.\ref{Fig:SUMAlgorithm}), and the last equality is due to
\eqref{A2}. A straightforward consequence of
\eqref{eq:decreasingobjval} is that the sequence of the objective
function values are non-increasing, that is
\begin{equation}\label{eq:decreasingobjvalf}
f(x^0) \geq f(x^1) \geq f(x^2) \geq \ldots
\end{equation}

Assume that there exists a subsequence $\{x^{r_j}\}$ converging to a
limit point~$z$. Then Assumptions \eqref{A1}, \eqref{A2} together
with \eqref{eq:decreasingobjvalf} imply that
\begin{align}
u(x^{r_{j + 1}},x^{r_{j+1}}) =  f(x^{r_{j+1}}) \leq f(x^{r_j + 1}) \leq u(x^{r_j +1} , x^{r_j}) \leq u(x , x^{r_j}), \quad \forall\; x \in \mathcal{X} \nonumber
\end{align}
Letting $j \rightarrow \infty$, we obtain
\[
u(z, z) \leq u(x,z), \quad \forall\; x \in \mathcal{X},
\]
which implies
\[
u'(x,z;d) \bigg|_{x = z} \geq  0, \quad\forall\; d\in \re^m \;\;{\rm with}\;\; z+ d \in \mathcal{X}.
\]
Combining with \eqref{A3}, we obtain
\[
f'(z;d) \geq 0,  \quad \forall\; d\in \re^m \;\;{\rm with}\; \;z+ d \in \mathcal{X},
\]
implying that $z$ is a stationary point of~$f(\cdot)$.
\end{proof}
\begin{coro}
Assume that the level set $\mathcal{X}^0 = \{x \mid f(x) \leq f(x^0)\}$ is compact and Assumption~\ref{AssumptionA} holds. Then, the sequence of iterates~$\{x^r\}$ generated by the SUM algorithm satisfy
\[
\lim_{r \rightarrow \infty} \quad d(x^r , \mathcal{X}^*) = 0,
\]
where~$\mathcal{X}^*$ is the set of stationary points of~\eqref{eq:OriginalProblem}.
\end{coro}
\begin{proof}
We prove the claim by contradiction. Suppose on the contrary that
there exists a subsequence $\{x^{r_j}\}$ such that
$d(x^{r_j},\mathcal{X}^*) \geq \gamma$ for some $\gamma>0$. Since
the sequence $\{x^{r_j}\}$ lies in the compact set $X^0$, it has a
limit point~$z$. By further restricting the indices of the
subsequence, we obtain
\[
d(z,\mathcal{X}^*)= \lim_{j\rightarrow \infty} d(x^{r_j},\mathcal{X}^*) \geq \gamma,
\]
which contradicts the fact that $z \in \mathcal{X}^*$ due to Theorem~\ref{thm:ConvergenceSUMConvex}.
\end{proof}

The above results show that under Assumption \ref{AssumptionA}, the SUM algorithm is globally convergent.
In the rest of this work, we derive similar results for a family of more
general inexact BCD algorithms. %In the next section, we combine the idea of
%SUM method with the block coordinate descent method.

\section{The Block Successive Upper-bound Minimization Algorithm}

In many practical applications, the optimization variables can be
decomposed into independent blocks. Such block structure, when
judiciously exploited, can lead to low-complexity algorithms that
are distributedly implementable. In this section, we introduce the
Block Successive Upper-bound Minimization (BSUM) algorithm,
which effectively takes such block structure into consideration.

%In many practical scenarios, a distributed algorithm for solving~\eqref{eq:OriginalProblem}
%is desired. Furthermore, the dimension of the problem could be large. Therefore, a coordinatewise
%approach is desired for solving~\eqref{eq:OriginalProblem}. In this section,
%we introduce the Coordinatewise Successive Upper-bound Minimization (BSUM) algorithm
%for solving~\eqref{eq:OriginalProblem}.\\

Let us assume that the feasible set $\mathcal{X}$ is the cartesian
product of $n$ closed convex sets: $\mathcal{X}= \mathcal{X}_1
\times \ldots \times \mathcal{X}_n$, with $\mathcal{X}_i \subseteq
\re^{m_i}$ and $\sum_i m_i = m$. Accordingly, the optimization
variable $x \in \re^m$ can be decomposed as: $x =
(x_1,x_2,\ldots,x_n)$, with $x_i \in \mathcal{X}_{i},
~i=1,\cdots,M$. We are interested in solving the problem
\begin{equation}\label{eq:OriginalProblemBlock}
\begin{split}
\min \quad & f(x)\\
\st \quad & x \in \mathcal{X}.
\end{split}
\end{equation}

Different from the SUM algorithm, the BSUM algorithm only updates
a single block of variables in each iteration. More precisely, at
iteration $r$, the selected block (say block $i$) is computed by
solving the following subproblem
\begin{equation}\label{eq:BUpperBound}
\begin{split}
\min_{x_i} \quad & u_i (x_i,x^{r-1})\\
\st \quad & x_i \in \mathcal{X}_i,
\end{split}
\end{equation}
where~$u_i(\cdot,x^{r-1})$ is again an approximation (in fact, a
global upper-bound) of the original objective $f(\cdot)$ at the
point $x^{r-1}$. Fig.~\ref{Fig:BSUMAlgorithm} summarizes the main
steps of the BSUM algorithm. Note that although the blocks are
updated following a simple cyclic rule, the algorithm and its
convergence results can be easily extended to the (more general)
essentially cyclic update rule as well. This point will be further
elaborated in Section \ref{sec:EssentiallyCyclic}.

\begin{figure}[t]
\centering
\begin{tabular}{|p{3.3in}|}
%\hline \textbf{input: $x_0,\alpha_k$, output: $x_k$}\\
\hline
\begin{itemize}
\item [1] \;Find a feasible point $x^0\in \mathcal{X}$ and set $r = 0$
\item [2] \; \textbf{repeat}
\item [3] \quad $r = r+1$, $i = (r \;{\rm mod} \;n)+1$
\item [4] \quad  Let $\mathcal{X}^{r} = \arg \min_{x_i\in \mathcal{X}_i} u_i(x_i,x^{r-1})$
\item [5] \quad Set $x_i^r$ to be an arbitrary element in~$\mathcal{X}^{r}$
\item [6] \quad Set $x_k^r = x_k^{r-1}, \quad \forall\; k \neq i$
\item [7] \; \textbf{until} some convergence criterion is met
\end{itemize}
\\
\hline
\end{tabular}\vspace{1.2em}
\caption{Pseudo code of the BSUM algorithm}
\label{Fig:BSUMAlgorithm} \vspace{-0.5cm}
\end{figure}

Now we are ready to study the convergence behavior of the BSUM
algorithm. To this end, the following regularity conditions on the
function $u_i(\cdot, \cdot)$ are needed.

\begin{assumption} \label{AssumptionB}
\begin{align}
& u_i(y_i, y) = f(y), \quad \forall\; y\in \mathcal{X}, \forall\; i \label{B1}\tag{B1} \\
& u_i(x_i,y) \geq f(y_1, \ldots, y_{i-1}, x_i ,y_{i+1}, \ldots, y_n),\quad\; \forall\; x_i \in \mathcal{X}_i, \forall\; y \in\mathcal{X}, \forall\; i\label{B2} \tag{B2}\\
& u_i' (x_i,y;d_i)\bigg|_{x_i = y_i} =  f' (y;d), \quad \forall\; d = (0,\ldots,d_i,\ldots,0) \;\;\st \;\; y_i + d_i \in \mathcal{X}_i, \forall\; i \label{B3}\tag{B3}\\
& u_i(x_i,y) \; {\rm is} \; {\rm continuous} \; {\rm in} \; (x_i,y) \tag{B4}, \quad \forall\; i
\end{align}
\end{assumption}

Similar to Proposition~\ref{lemma:SUMAssumption}, we can identify a
sufficient condition to ensure \eqref{B3}.
\begin{prop} \label{lemma:BSUMAssumption}
Assume $f(x) = f_0(x) + f_1(x)$, where $f_0(\cdot)$ is continuously differentiable and the directional derivative of $f_1(\cdot)$ exists at every point~$x\in \mathcal{X}$. Consider $u_i(x_i,y) = u_{0,i}(x_i,y) + f_1(x)$, where $u_{0,i}(x_i,y)$ satisfies the following assumptions
\begin{align}
& u_{0,i}(x_i, x) = f_0(x),\quad \forall\; x \in \mathcal{X}, \quad\forall\; i \nonumber \\
& u_{0,i}(x_i,y) \geq f_0(y_1,\ldots, y_{i-1},x_i,y_{i+1},\ldots,y_n),\; \forall\; x,y \in \mathcal{X}\quad\forall\; i. \nonumber
\end{align}
Then, \eqref{B1}, \eqref{B2}, and \eqref{B3} hold.
\end{prop}
\begin{proof}
The proof is exactly the same as the proof in Proposition~\ref{lemma:SUMAssumption}.
\end{proof}

%Theorem \ref{thm:ConvergenceBSUMConvex} summarizes the convergence
%results of BSUM method.
The convergence results regarding to the BSUM algorithm consist of
two parts. In the first part, a quasi-convexity of the objective
function is assumed, which guarantees the existence of the limit
points. This is in the same spirit of the classical proof of
convergence for the BCD method in \cite{Bertsekas_Book_Nonlinear}.
However, if we know that the iterates lie in a compact set, then a
stronger result can be proved. Indeed, in the second part of the
theorem, the convergence is obtained by relaxing the quasi-convexity
assumption while imposing the compactness assumption of level sets.

\begin{thm}\label{thm:ConvergenceBSUMConvex}
\begin{itemize}
\item[]
\item[(a)] Suppose that the function $u_i(x_i,y)$ is quasi-convex in~$x_i$ and Assumption~\ref{AssumptionB} holds. Furthermore, assume that the subproblem~\eqref{eq:BUpperBound} has a unique solution for any point~$x^{r-1} \in \mathcal{X}$. Then, every limit point~$z$ of the iterates generated by the BSUM algorithm is a coordinatewise minimum of~\eqref{eq:OriginalProblemBlock}. In addition, if $f(\cdot)$ is regular at~$z$, then $z$ is a stationary point of~\eqref{eq:OriginalProblemBlock}.
\item [(b)] Suppose the level set $\mathcal{X}^0 = \{x \mid f(x) \leq f(x^0)\}$ is compact and Assumption~\ref{AssumptionB} holds. Furthermore, assume that the subproblem~\eqref{eq:BUpperBound} has a unique solution for any point $x^{r-1} \in \mathcal{X}$ for at least $n-1$ blocks. If $f(\cdot)$ is regular at every point in the set of stationary points~$\mathcal{X}^*$ with respect to the coordinates $x_1,\ldots,x_n$. Then, the iterates generated by the BSUM algorithm converge to the set of stationary points, i.e.,
\[
\lim_{r \rightarrow \infty} \quad d(x^r,\mathcal{X}^*) = 0.
\]
\end{itemize}
\end{thm}
\begin{proof}
The proof of part (a) is similar to the one in \cite{Bertsekas_Book_Nonlinear} for block coordinate descent approach. First of all, since a locally tight upper bound of~$f(\cdot)$ is minimized at each iteration, we have
\begin{equation}\label{eq:DecreasingObjFnThm}
f(x^0) \geq f(x^1) \geq f(x^2)\geq \ldots.
\end{equation}
Therefore, the continuity of $f(\cdot)$ implies
\begin{equation} \label{eq:fxgoesfz}
\lim_{r \rightarrow \infty}f(x^{r}) = f(z).
\end{equation}

Let us consider the subsequence $\{x^{r_j}\}$ converging to the limit point $z$. Since the number of blocks is finite, there exists a block which is updated infinitely often in the subsequence $\{r_j\}$. Without loss of generality, we assume that block~$n$ is updated infinitely often. Thus, by further restricting to a subsequence, we can write
\begin{equation}
x_n^{r_j}  = \arg \min_{x_n}  \quad u_n (x_n, x^{r_j-1}). \nonumber
\end{equation}
Now we prove that $x^{r_j + 1} \rightarrow z$, in other words, we
will show that $x_1^{r_j +1} \rightarrow z_1$. Assume the contrary
that $x_1^{r_j+1}$ does not converge to $z_1$. Therefore by further
restricting to a subsequence, there exists $\bar\gamma > 0$ such
that
\[
\bar\gamma \leq \gamma^{r_j} = \|x_1^{r_j+1} - x_1^{r_j}\|,\
\forall~r_j.
\]
Let us normalize the difference between $x_1^{r_j}$ and $x_1^{r_j+1}$, i.e.,
\[
s^{r_j} \triangleq \frac{x_1^{r_j+1} - x_1^{r_j}}{\gamma^{r_j}}.
\]
Notice that $\|s^{r_j}\|=1$, thus $s^{r_j}$ belongs to a compact set and it has a limit point $\bar{s}$. By further restricting to a subsequence that converges to~$\bar{s}$, using \eqref{B1} and \eqref{B2}, we obtain
\begin{align}
f(x^{r_j+1}) &\leq u_1(x_1^{r_j+1},x^{r_j}) \label{TempTheomConvBert1}\\
& = u_1(x_1^{r_j} + \gamma^{r_j}s^{r_j} , x^{r_j}) \label{TempTheomConvBert2}\\
& \leq u_1(x_1^{r_j} + \epsilon \bar\gamma s^{r_j} , x^{r_j}), \quad  \forall\; \epsilon \in [0,1] \label{TempTheomConvBert3}\\
& \leq u_1(x^{r_j}_1,x^{r_j}) \label{TempTheomConvBert4}\\
& = f(x^{r_j}) \label{TempTheomConvBert5},
\end{align}
where \eqref{TempTheomConvBert1} and \eqref{TempTheomConvBert5} hold due to \eqref{B1} and \eqref{B2}. The inequalities \eqref{TempTheomConvBert3} and \eqref{TempTheomConvBert4} are the result of quasi-convexity of $u(\cdot,x^{r_j})$. Letting $j \rightarrow \infty$ and combining \eqref{TempTheomConvBert1}, \eqref{TempTheomConvBert3}, \eqref{eq:fxgoesfz}, and \eqref{TempTheomConvBert5} imply
\[
f(z)\leq u_1(z_1 + \epsilon \bar\gamma \bar{s} ,z) \leq f(z), \quad \forall\; \epsilon \in [0,1],
\]
or equivalently
\begin{equation} \label{TempTheomConvBert6}
f(z) = u_1(z_1 + \epsilon \bar\gamma \bar{s},z),  \quad \forall\; \epsilon \in [0,1].
\end{equation}
Furthermore,
\begin{align}
u_1(x_1^{r_{j+1}} ,x^{r_{j+1}}) &= f(x^{r_{j+1}}) \leq f(x^{r_j+1}) \nonumber\\
&\leq u_1(x_1^{r_j+1}, x^{r_j}) \leq u_1(x_1,x^{r_j}), \quad \forall\; x_1 \in \mathcal{X}_1.\nonumber
\end{align}
Letting $j \rightarrow \infty$, we obtain
\[
u_1(z_1,z) \leq u_1(x_1,z), \quad \forall\; x_1 \in \mathcal{X}_1,
\]
which further implies that $z_1$ is the minimizer of $u_1(\cdot,z)$. On the other hand, we assume that the minimizer is unique, which contradicts \eqref{TempTheomConvBert6}. Therefore, the contrary assumption is not true, i.e., $x^{r_j+1} \rightarrow z$.\\

Since $x_1^{r_j+1} = \arg \min_{x_1\in \mathcal{X}_1} \quad u_1(x_1, x^{r_j})$, we get
\[
u_1(x_1^{r_j+1} , x^{r_j}) \leq u_1(x_1 , x^{r_j}) \quad \forall\; x_1 \in \mathcal{X}_1.
\]
Taking the limit $j \rightarrow \infty$ implies
\[
u_1(z_1 , z) \leq u_1(x_1 , z) \quad \forall\; x_1 \in \mathcal{X}_1,
\]
which further implies
\[
 u_1'(x_1 , z;d_1)\bigg|_{x_1 = z_1} \geq 0, \quad \forall\; d_1 \in \mathbb{R}^{m_1} \quad{\rm with} \quad z_1+d_1 \in \mathcal{X}_1.
\]
Similarly, by repeating the above argument for the other blocks, we obtain
\begin{equation} \label{eq:DirDerivOptiUTheom}
u'_k (x_k , z;d_k)\bigg|_{x_k = z_k} \geq 0, \quad \forall\; d_k \in \mathbb{R}^{m_k} \quad{\rm with} \quad d_k + z_k \in \mathcal{X}_k, \quad \quad \forall\; k=1,\ldots, n.
\end{equation}
Combining \eqref{B3} and \eqref{eq:DirDerivOptiUTheom} implies
\begin{equation}
  f'(z;d) \geq 0,\quad \forall\; d=(0,\ldots,d_k,\ldots,0)  \quad\st \quad d + z \in \mathcal{X}, \;\forall\; k\nonumber
\end{equation}
in other words, $z$ is the coordinatewise minimum  of $f(\cdot)$.\\

Now we prove part (b) of the theorem. Without loss of generality, let us assume that \eqref{eq:BUpperBound} has a unique solution at every point $x^{r-1}$ for $i=1,2,\ldots,n-1$. Since the iterates lie in a compact set, we only need to show that every limit point of the iterates is a stationary point of~$f(\cdot)$. To do so, let us consider a subsequence~$\{x^{r_j}\}$ which converges to a limit point~$z \in \mathcal{X}^0 \subseteq \mathcal{X}$. Since the number of blocks is finite, there exists a block~$i$ which is updated infinitely often in the subsequence $\{x^{r_j}\}$. By further restricting to a subsequence, we can assume that
\[
x_i^{r_j} \in \arg\min_{x_i} \quad u_i(x_i, x^{r_j-1}).
\]
Since all the iterates lie in a compact set, we can further restrict to a subsequence such that
\[
\lim_{j \rightarrow \infty}\quad x^{r_j-i + k}  = z^k, \quad  \forall\; k=0,1,\ldots,n,
\]
where $z^k \in \mathcal{X}^0 \subseteq \mathcal{X}$ and $z^i = z$. Moreover, due to the update rule in the algorithm, we have
\[
u_k(x_k^{r_j -i + k}, x^{r_j - i + k -1}) \leq u_k(x_k, x^{r_j - i + k-1}), \quad \forall\; x_k \in \mathcal{X}_k, \quad\quad k=1,2,\ldots, n.
\]
Taking the limit~$j \rightarrow \infty$, we obtain
\begin{equation} \label{eq:udirectionalmin}
u_k(z_k^k , z^{k-1}) \leq u_k(x_k , z^{k-1}), \quad \forall\; x_k \in \mathcal{X}_k,\quad \quad k=1,2,\ldots, n.
\end{equation}
Combining \eqref{eq:udirectionalmin}, \eqref{B1} and \eqref{B2} implies
\begin{equation}\label{eq:boundingu}
f(z^k) \leq u_k (z_k^k,z^{k-1}) \leq u_k(z_k^{k-1}, z^{k-1}) = f(z^{k-1}),\quad  k = 1,\ldots,n.
\end{equation}
On the other hand, the objective function is non-increasing in the algorithm and it has a limit. Thus,
\begin{equation}\label{eq:limitObjThm2}
f(z^0) = f(z^1) = \ldots = f(z^n).
\end{equation}
Using \eqref{eq:boundingu}, \eqref{eq:limitObjThm2}, and \eqref{eq:udirectionalmin}, we obtain
\begin{equation}\label{eq:temp1Thm2}
f(z) = u_k (z_k^k ,z^{k-1}) \leq u_k (x_k, z^{k-1}), \; \forall\; x_k \in \mathcal{X}_k, \quad k = 1,2,\ldots, n.
\end{equation}
Furthermore, $f(z) = f(z^{k-1}) = u_k(z_k^{k-1}, z^{k-1})$ and therefore,
\begin{equation}\label{eq:temp2Thm2}
u_k (z_k^{k-1} ,z^{k-1}) \leq u_k (x_k, z^{k-1}), \; \forall\; x_k \in \mathcal{X}_k, \quad k = 1,2,\ldots, n.
\end{equation}
The inequalities~$\eqref{eq:temp1Thm2}$ and \eqref{eq:temp2Thm2} imply that $z_k^{k-1}$ and $z_k^k$ are both the minimizer of $u_k(\cdot,z^{k-1})$. However, according to our assumption, the minimizer is unique for $k = 1,2,\ldots, n-1$ and therefore,
\[
z^0 = z^1 = z^2 = \ldots =z^{n-1} = z
\]
Plugging the above relation in \eqref{eq:udirectionalmin} implies
\begin{equation}\label{eq:temp3Thm2}
u_k (z_k ,z) \leq u_k (x_k, z), \; \forall\; x_k \in \mathcal{X}_k, \quad k = 1,2,\ldots, n-1.
\end{equation}
Moreover, by setting $k=n$ in \eqref{eq:temp2Thm2}, we obtain
\begin{equation}\label{eq:temp4Thm2}
u_n (z_n ,z) \leq u_n (x_n, z), \quad\forall\; x_n \in \mathcal{X}_n.
\end{equation}
The inequalities~\eqref{eq:temp3Thm2} and \eqref{eq:temp4Thm2} imply that
\[
 u'_k(x_k,z;d_k)\bigg|_{x_k = z_k} \geq 0, \quad \forall\; d_k\in \re^{m_k}\; {\rm with} \;z_k + d_k \in \mathcal{X}_k, \quad k=1,2,\ldots,n.
\]
Combining this with \eqref{B3} yields
\[
 f'(z;d) \geq 0, \quad \forall\; d = (0,\ldots,d_k,\ldots,0) \;\;{\rm with} \;\;z_k + d_k \in \mathcal{X}_k, \quad k = 1,2,\ldots,n,
\]
which implies the stationarity of the point~$z$ due to the regularity of $f(\cdot)$.
\end{proof}

The above result extends the existing result of block coordinate descent method \cite{Bertsekas_Book_Nonlinear}
and \cite{BCDTseng} to the BSUM case where only an approximation of the objective function
is minimized at each iteration. As we will see in Section \ref{sec:Applications},
our result implies the global convergence of several existing algorithms including the EM algorithm
or the DC method when the Gauss-Seidel update rule is used.\\

\section{The Maximum Improvement Successive Upper-bound Minimization
Algorithm}\label{sec:MISUM}

A key assumption for the BSUM algorithm is the uniqueness of the
minimizer of the subproblem. This assumption is necessary even for
the simple BCD method \cite{Bertsekas_Book_Nonlinear}. In general,
by removing such assumption, the convergence is not guaranteed (see
\cite{PowellBCD73} for examples) unless we assume pseudo convexity
in pairs of the variables \cite{ZadehBCDPsudoConvexity},
\cite{BCDTseng}. In this section, we explore the possibility of
removing such uniqueness assumption.

Recently, Chen \textit{et al.} \cite{MBIChen2012} have proposed a
related Maximum Block Improvement (MBI) algorithm, which differs
from the conventional BCD algorithm only by its update schedule.
More specifically, only the block that provides the {\it maximum
improvement} is updated at each step. Remarkably, by utilizing such
modified updating rule (which is similar to the well known
Gauss-Southwell update rule), the per-block subproblems are allowed
to have multiple solutions. Inspired by this recent development, we
propose to modify the BSUM algorithm similarly by simply updating
the block that gives the maximum improvement. We name the resulting
algorithm the Maximum Improvement Successive Upper-bound
Minimization (MISUM) algorithm, and list its main steps in
Fig.~\ref{Fig:MISUMAlgorithm}.

\begin{figure*}[t]
\centering
\begin{tabular}{|p{3.3in}|}
%\hline \textbf{input: $x_0,\alpha_k$, output: $x_k$}\\
\hline
\begin{itemize}
\item [1] \;Find a feasible point $x^0\in \mathcal{X}$ and set $r = 0$
\item [2] \; \textbf{repeat}
\item [3] \quad $r = r+1$
\item [4] \quad Let $k = \arg \min_i \min_{x_i} u_i(x_i,x^{r-1})$
\item [5] \quad  Let $\mathcal{X}^{r} = \arg \min_{x_k\in \mathcal{X}_k} u_k(x_k,x^{r-1})$
\item [6] \quad Set $x_k^r$ to be an arbitrary element in~$\mathcal{X}^{r}$
\item [7] \quad Set $x_i^r = x_i^{r-1}, \quad \forall\; i \neq k$
\item [8] \; \textbf{until} some convergence criterion is met
\end{itemize}
\\
\hline
\end{tabular}\vspace{1.2em}
\caption{Pseudo code of the MISUM algorithm}
\label{Fig:MISUMAlgorithm} \vspace{-0.5cm}
\end{figure*}

Clearly the MISUM algorithm is more general than the MBI method
proposed in \cite{MBIChen2012}, since only an approximate version
of the subproblem is solved at each iteration.
Theorem~\ref{thm:ConvergenceMISUM} states the convergence result for
the proposed MISUM algorithm.
\begin{thm}\label{thm:ConvergenceMISUM}
Suppose that Assumption~\ref{AssumptionB} is satisfied. Then, every limit point~$z$ of the iterates generated by the MISUM algorithm is a coordinatewise minimum of~\eqref{eq:OriginalProblemBlock}. In addition, if $f(\cdot)$ is regular at~$z$, then $z$ is a stationary point of~\eqref{eq:OriginalProblemBlock}.
\end{thm}
\begin{proof}
Let us define $R_i(y)$ to be the minimum objective value of the
$i$-th subproblem at a point~$y$, i.e.,
\[
R_i(y) \triangleq \min_{x_i} \quad u_i(x_i,y).
\]
Using a similar argument as in Theorem
\ref{thm:ConvergenceBSUMConvex}, we can show that the sequence of
the objective function values are non-increasing, that is
\[
f(x^r) = u_i(x_i^r,x^r) \geq R_i (x^r) \geq f(x^{r+1}).
\]
Let $\{x^{r_j}\}$ be the subsequence converging to a limit point~$z$. For every
fixed block index $i=1,2,\ldots,n$ and every $x_i \in
\mathcal{X}_i$, we have the following series of inequalities
\begin{align}
u_i (x_i,x^{r_j}) &\geq R_i(x^{r_j})\nonumber\\
&\geq u_k(x_k^{r_j+1},x^{r_j}) \nonumber\\
&\geq f(x^{r_{j}+1}) \nonumber\\
& \geq f(x^{r_{j+1}}) \nonumber\\
& = u_i (x_i^{r_{j+1}} , x^{r_{j+1}}), \nonumber
\end{align}
where we use $k$ to index the block that provides the maximum
improvement at iteration $r_j+1$. The first and the second
inequalities are due to the definition of the function $R_i(\cdot)$
and the MISUM update rule, respectively. The third inequality is
implied by the upper bound assumption \eqref{B2}, while the last
inequality is due to the non-increasing property of the objective values.

Letting $j\rightarrow \infty$, we obtain
\[
u_i(x_i,z) \geq u_i(z_i,z), \quad \forall\; x_i \in \mathcal{X}_i, \quad  i=1,2,\ldots,n.
\]
%which implies the inequality should hold with equality since the other direction holds according to the definition of $R_i(\cdot)$. Thus,
%\[
% u_i(z_i,z) = u_i(R_i(z),z) \leq u_i(x_i,z), \quad \forall\; x_i \in \mathcal{X}_i.
%\]
The first order optimality condition implies
\[
u'_i(x_i,z;d_i)\bigg|_{x_i = z_i} \geq 0, \quad \forall\; d_i \;\;{\rm with}\;\;z_i + d_i \in \mathcal{X}_i, \quad \forall\; i = 1,2,\ldots,n.
\]
Combining this with \eqref{B3} yields
\[
f'(z;d)\geq 0, \quad \forall\; d = (0,\ldots,d_i,\ldots,0) \;\;{\rm with} \;\;z_i + d_i \in \mathcal{X}_i, \quad i = 1,2,\ldots,n.
\]
In other words, $z$ is the coordinatewise minimum of $f(\cdot)$.
\end{proof}

The main advantage of the MISUM algorithm over the BSUM algorithm is
that its convergence does not rely on the uniqueness of the
minimizer for the subproblems. On the other hand, each iteration of
MISUM algorithm is more expensive than the BSUM since the
minimization needs to be performed for all the blocks. Nevertheless,
the MISUM algorithm is more suitable when parallel processing units
are available, since the minimizations with respect to all the
blocks can be carried out simultaneously.

\section{Successive Convex Approximation of a Smooth Function}
\label{sec:BSCA}

In the previous sections, we have demonstrated that the stationary
solutions of the problems \eqref{eq:OriginalProblem} and
\eqref{eq:OriginalProblemBlock} can be obtained by successively
minimizing a sequence of upper-bounds of $f(\cdot)$. However, in
practice, unless the objective $f(\cdot)$ possesses certain
convexity/concavity structure, those upper-bounds may not be easily
identifiable. In this section, we extend the BSUM algorithm by
further relaxing the requirement that the approximation functions $\{u_i(x_i,y)\}$
must be the global upper-bounds of the original objective
$f$.

%A natural question to ask about the BSUM algorithm is that what
%happens if the approximation function~$u_i(\cdot,x^r)$ is not
%necessarily a global upper bound of $f(\cdot)$.
%
%In fact, at iteration~$r$, we can update block $i = (r \;{\rm mod}\;
%n) + 1$ by minimizing the convex approximation function~$h_i (x_i,
%x^r)$ which satisfies the following assumption

Throughout this section, we use $h_i (., .)$ to denote the convex
approximation function for the $i$th block. Suppose that $h_i(x_i,
x)$ is no longer a global upper-bound of $f(x)$, but only a first
order approximation of $f(x)$ at each point, i.e.,
\begin{equation}
h'_i (y_i , x;d_i)\bigg|_{y_i=x_i} =  f'(x;d), \quad   \;\forall\; d =
(0, \ldots, d_i, \ldots,0)\quad {\rm with}\;\; \;x_i + d_i \in
\mathcal{X}_i.\label{assptnNotubd}
\end{equation}
In this case, simply optimizing the approximate functions in each
step may not even decrease the objective function. Nevertheless, the
minimizer obtained in each step can still be used to construct a
good search direction, which, when combined with a proper step size
selection rule, can yield a sufficient decrease of the objective value.

Suppose that at iteration $r$, the $i$-th block needs to be updated.
Let $y^r_i\in\min_{y_i\in\mathcal{X}_i} h_i(y_i,x^{r-1})$ denote the
optimal solution for optimizing the $i$-th approximation function at
the point $x^{r-1}$. We propose to use $y^{r}_i-x^{r-1}_i$ as the
search direction, and adopt the Armijo rule to guide the step size
selection process. We name the resulting algorithm the
Block Successive Convex Approximation (BSCA) algorithm. Its
main steps are given in Figure~\ref{Fig:BSCAalgorithm}. %, while the
%step size selection part is described separately in
%Figure~\ref{Fig:Armijo}.

\begin{figure*}[t]
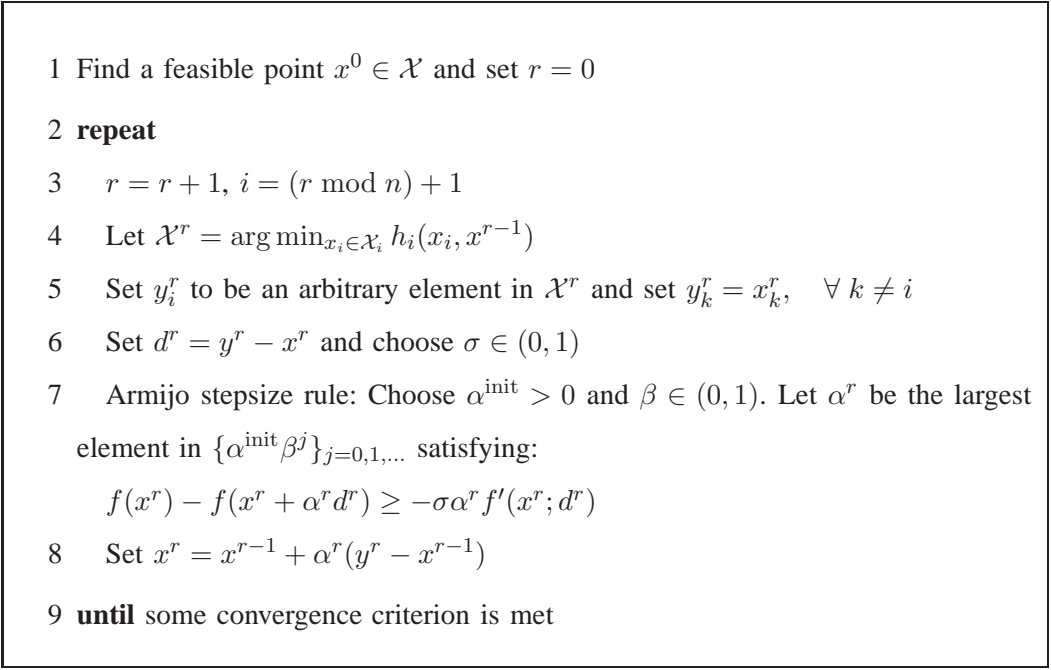

\centering
\begin{tabular}{|p{5.3in}|}
%\hline \textbf{input: $x_0,\alpha_k$, output: $x_k$}\\
\hline
\begin{itemize}
\item [1] \;Find a feasible point $x^0\in \mathcal{X}$ and set $r = 0$
\item [2] \; \textbf{repeat}
\item [3] \quad $r = r+1$, $i = (r \;{\rm mod} \;n)+1$
\item [4] \quad  Let $\mathcal{X}^{r} = \arg \min_{x_i\in \mathcal{X}_i} h_i(x_i,x^{r-1})$
\item [5] \quad Set $y_i^r$ to be an arbitrary element in~$\mathcal{X}^{r}$ and set $y_k^r = x_k^r, \quad \forall\; k \neq i$
\item [6] \quad Set $d^r = y^r - x^r$ and choose $\sigma\in (0,1)$
\item [7] \quad Armijo stepsize rule: Choose $\alpha^{\rm init} >0$ and $\beta \in (0,1)$. Let $\alpha^r$ be the largest element in $\{\alpha^{\rm init} \beta^j\}_{j = 0,1,\ldots}$ satisfying:
\item[] \quad $f(x^r) - f(x^r + \alpha^r d^r) \geq - \sigma \alpha^r  f'(x^r;d^r)$
\item [8] \quad Set $x^r = x^{r-1} + \alpha^r (y^r - x^{r-1})$
\item [9] \; \textbf{until} some convergence criterion is met \vspace*{-10pt}
\end{itemize}
\\%[-10pt]
\hline
\end{tabular}%\vspace{.2em}
\caption{Pseudo code of the BSCA algorithm}
\label{Fig:BSCAalgorithm}
\end{figure*}
%\vspace{-0.5cm}
%
%\begin{figure*}[t]
%\centering
%\begin{tabular}{|p{3.3in}|}
%%\hline \textbf{input: $x_0,\alpha_k$, output: $x_k$}\\
%\hline
%\begin{itemize}
%\item [1] Set $d^r = y^r - x^r$ and choose $\sigma\in (0,1)$
%\item [2] Choose $\alpha^{\rm init} >0$. Let $\alpha^r$ be the largest element in $\{\alpha^{\rm init} \beta^j\}_{j = 0,1,\ldots}$ satisfying:
%\item[] \quad $f(x^r) - f(x^r + \alpha^r d^r) \geq - \sigma \alpha^r  f'(x^r;d^r)$
%\end{itemize}
%\\
%\hline
%\end{tabular}\vspace{1.2em}
%\caption{Armijo  step size selection rule} \label{Fig:Armijo}
%\vspace{-0.5cm}
%\end{figure*}

Note that for $d^r=(0,\ldots,d_i^r,\ldots,0)$ with $d_i^r = y_i^r - x_i^r$, we have
\begin{equation}
 f'(x^r;d^r)  =  h_i' (x_i, x^r;d_i^r)\bigg|_{x_i = x_i^r} = \lim_{\lambda \downarrow 0} \frac{h_i (x_i^r + \lambda d_i^r , x^r) -h_i(x_i^r,x^r)}{\lambda} \leq 0, \label{eq:negativeDirection}
\end{equation}
where the inequality is due to the fact that $h_i(\cdot)$ is convex
 and $y_i^r = x_i^r + d_i^r$ is the minimizer at iteration~$r$.
Moreover, there holds
\[
f(x^r) - f(x^r + \alpha d^r) = - \alpha  f'(x^r;d^r) + o(\alpha), \quad \forall\; \alpha>0.
\]
Hence the Armijo step size selection rule in Figure~\ref{Fig:BSCAalgorithm}
is well defined when $ f'(x^r;d^r) \neq 0$, and there exists $j\in
\{0,1,\ldots\}$ such that  for $\alpha^r = \alpha^{\rm init}
\beta^j$,
\begin{equation}\label{eq:decreaseThmArmijo}
f(x^r) - f(x^r + \alpha^r d^r) \geq - \sigma \alpha^r  f'(x^r;d^r).
\end{equation}

 The following theorem states the
convergence result of the proposed algorithm.
\begin{thm}\label{thm:ConvergenceBSCA}
Suppose that $f(\cdot)$ is continuously differentiable and that
Assumption~\eqref{assptnNotubd} holds. Furthermore, assume
that $h(x,y)$ is strictly convex in $x$ and continuous in $(x,y)$.
Then every limit point of the iterates generated by the BSCA algorithm is a stationary point
of~\eqref{eq:OriginalProblem}.
\end{thm}
\begin{proof}
First of all,  due to the use of Armijo step size selection rule, we have
\[
f(x^r) - f(x^{r+1})  \geq - \sigma \alpha^r  f'(x^r;d^r) \geq 0,
\]
which implies
\begin{equation}\label{eq:tempThmArmijo1}
\lim_{r \rightarrow \infty} \alpha^r f'(x^r;d^r) = 0.
\end{equation}
Consider a limit point $z$ and a subsequence $\{x^{r_j}\}_j$ converging to $z$. Since $\{f(x^r)\}$ is a monotonically decreasing sequence, it follows that
%\begin{equation} \label{eq:decrease}
%f(x^{r+1}) \leq f(x^r), \quad \forall\; r,
%\end{equation}
%due to \eqref{eq:negativeDirection} and \eqref{eq:decreaseThmArmijo}.  Clearly, the equation \eqref{eq:decrease} implies
\[
\lim_{r \rightarrow \infty} \; f(x^r) = f(z).
\]
%Moreover,
By further restricting to a subsequence if necessary, we can assume without loss of generality that in the subsequence $\{x^{r_j}\}_j$ the first block is updated. %Let us restrict ourselves to the subsequence that the first block is updated.
We first claim that we can further restrict to a further subsequence such that
\begin{equation}\label{eq:tempThmArmijo5}
\lim_{j \rightarrow \infty } d^{r_j}  = 0.
\end{equation}
We prove this by contradiction. Let us assume the contrary so that there exists $\delta, \; 0<\delta<1$ and $\ell \in \{1,2,\ldots\}$
\begin{equation}\label{eq:ContraryThmArmijo}
\|d^{r_j}\| \geq \delta,\;\forall\; j \geq \ell.
\end{equation}
Defining $p^{r_j} = \frac{d^{r_j}}{\|d^{r_j}\|}$, the equation~\eqref{eq:tempThmArmijo1} implies $\alpha^{r_j} \|d^{r_j}\|  f'(x^{r_j};p^{r_j}) \rightarrow 0$. Thus, we have the following two cases:\\
\textbf{Case A:} $  f'(x^{r_j};p^{r_j}) \rightarrow 0$ along a subsequence of $\{x^{r_j}\}$. Let us restrict ourselves to that subsequence. Since $\|p^{r_j}\|=1$, there exists a limit point~$\bar{p}$. By further restricting to a subsequence and using the smoothness of~$f(\cdot)$, we obtain
\begin{equation}\label{eq:tempThmArmijo2}
 f'(z;\bar{p}) = 0.
\end{equation}
Furthermore, due to the strict convexity of $h_1(\cdot,z)$,
\begin{equation}\label{eq:tempThmArmijo3}
h_1(z_1 + \delta \bar{p}_1, z) > h_1(z_1,z) + \delta  h'_1(x_1,z;\bar{p}_1)\bigg|_{x_1 = z_1}\geq h_1 (z_1,z),
\end{equation}
where $\bar{p}_1$ is the first block of $\bar{p}$ and the last step is due to \eqref{eq:tempThmArmijo2} and \eqref{assptnNotubd}. On the other hand, since $x_1^{r_j} + \delta p_1^{r_j}$ lies between $x_1^{r_j}$ and $y_1^{r_j}$, we have
\[
h_1(x_1^{r_j} + \delta p_1^{r_j},x^{r_j}) \leq h_1(x_1^{r_j},x^{r_j}).
\]
Letting $j \rightarrow \infty$ along the subsequence, we obtain
\begin{equation}\label{eq:tempThmArmijo4}
h_1(z_1 + \delta \bar{p}_1, z) \leq h_1 (z_1,z),
\end{equation}
which contradicts \eqref{eq:tempThmArmijo3}.\\
\textbf{Case B:} $\alpha^{r_j} \|d^{r_j}\| \rightarrow 0$ along a subsequence. Let us restrict ourselves to that subsequence. Due to the contrary assumption~\eqref{eq:ContraryThmArmijo},
\[
\lim_{j \rightarrow \infty}\alpha^{r_j} = 0,
\]
which further implies that there exists $j_0 \in \{1,2,\ldots\}$  such that
\[
f(x^{r_j} + \frac{\alpha^{r_j}}{\beta} d^{r_j}) - f(x^{r_j}) > \sigma \frac{\alpha^{r_j}}{\beta}  f'(x^{r_j};d^{r_j}), \;\quad \forall\; j \geq j_0.
\]
Rearranging the terms, we obtain
\[
\frac{f(x^{r_j} + \frac{\alpha^{r_j}}{\beta} \|d^{r_j}\| p^{r_j}) - f(x^{r_j})}{\frac{\alpha^{r_j}}{\beta} \|d^{r_j}\|} > \sigma f'(x^{r_j};p^{r_j}), \quad \forall\; j \leq j_0.
\]
Letting $j\rightarrow \infty$ along the subsequence that $p^{r_j} \rightarrow \bar{p}$, we obtain
\[
 f'(z;\bar{p}) \geq \sigma  f'(z;\bar{p}),
\]
which implies $f(z;\bar{p})  \geq 0$ since $\sigma<1$. Therefore,
using an argument similar to the previous case,
\eqref{eq:tempThmArmijo3} and \eqref{eq:tempThmArmijo4} hold, which
 is a contradiction. Thus, the assumption \eqref{eq:ContraryThmArmijo} must be false and the condition \eqref{eq:tempThmArmijo5} must hold.
%\begin{equation}\label{eq:tempThmArmijo5}
%\lim_{j \rightarrow \infty} d^{r_j} = 0.
%\end{equation}
On the other hand, $y_1^{r_j}$ is the minimizer of $h_1(\cdot,x^{r_j})$; thus,
\begin{equation}\label{eq:tempThmArmijo6}
h_1(y_1^{r_j},x^{r_j}) \leq h_1 (x_1,x^{r_j}), \;\quad  \forall\; x_1 \in \mathcal{X}_1.
\end{equation}
Note that $y_1^{r_j} = x_1^{r_j} + d_1^{r_j}$. Combining \eqref{eq:tempThmArmijo5} and \eqref{eq:tempThmArmijo6} and letting $j \rightarrow \infty$ yield
\[
h_1 (z_1,z) \leq h_1(x_1, z), \;\quad \forall\; x_1 \in \mathcal{X}_1.
\]
The first order optimality condition and assumption~\eqref{assptnNotubd} imply
\[
f'(z;d) \geq 0, \; \forall\; d = (d_1, 0, \ldots,0) \quad\;{\rm with} \quad\; z_1 + d_1 \in \mathcal{X}_1.
\]
On the other hand, since $d^{r_j} \rightarrow 0$, it follows that
\[
\lim_{j\rightarrow \infty} x^{r_j +1} = z.
\]
Therefore, by restricting ourselves to the subsequence that $d^{r_j} \rightarrow 0$ and repeating the above argument $n$ times, we obtain
\[
f'(z;d) \geq 0,\quad \forall\; d = (0,\ldots, d_k, \ldots,0) \quad \;{\rm with} \quad\; z_k + d_k \in \mathcal{X}_k; \;  k =1,\ldots,n.
\]
Using the regularity of $f(\cdot)$ at point~$z$ completes the proof.
\end{proof}

We remark that the proposed BSCA method is related to the coordinate
gradient descent method \cite{CGDTseng}, in which a strictly convex
second order approximation of the objective function is minimized at
each iteration. It is important to note that the convergence results
of these two algorithm do not imply each other. The BSCA algorithm,
although more general in the sense that the approximation function
could take the form of any strictly convex function that
satisfies~\eqref{assptnNotubd}, only covers the case when the
objective function is smooth. Nevertheless, the freedom provided by
the BSCA to choose a more general approximation function allows one
to better approximate the original function at each iteration.

\section{Overlapping Essentially Cyclic Rule}
\label{sec:EssentiallyCyclic}

In both the BSUM and the BSCA algorithms considered in the previous sections,
variable blocks are updated in a simple cyclic manner. In
this section, we consider a very general block scheduling rule
named the overlapping essentially cyclic rule and show they still ensure the convergence
of the BSUM and the BSCA algorithms.

In the so called overlapping essentially cyclic rule, at each iteration $r$, a group
$\vartheta^r$ of the variables is chosen to be updated where
\[
\vartheta^r \subseteq \{1,2,\ldots, n\} \quad {\rm and} \quad \vartheta^r \neq \emptyset.
\]
Furthermore, we assume that the update rule is essentially cyclic with period $T$, i.e.,
\[
\bigcup_{i=1}^T \vartheta^{r+i} = \{1,2,\ldots,n\}, \quad \forall\; r.
\]
Notice that this update rule is more general than the essentially cyclic rule since the blocks are allowed to
have overlaps. Using the overlapping essentially cyclic update rule,
almost all the convergence results presented so far still hold. For
example, the following corollary extends the convergence of BSUM to
the overlapping essentially cyclic case.

\begin{coro}\label{coro:ConvergenceBSUMConvex}
\begin{itemize}
\item[]
\item[(a)] Assume that the function $u_i(x_i,y)$ is quasi-convex in~$x_i$ and Assumption~\ref{AssumptionB} is satisfied. Furthermore, assume that the overlapping essentially cyclic update rule is used and  the subproblem~\eqref{eq:BUpperBound} has a unique solution for every block $\vartheta^r$. Then, every limit point~$z$ of the iterates generated by the BSUM algorithm is a coordinatewise minimum of~\eqref{eq:OriginalProblemBlock}. In addition, if $f(\cdot)$ is regular at~$z$ with respect to the updated blocks, then $z$ is a stationary point of~\eqref{eq:OriginalProblemBlock}.
\item[(b)] Assume the level set $\mathcal{X}^0 = \{x \mid f(x) \leq f(x^0)\}$ is compact and Assumption~\ref{AssumptionB} is satisfied. Furthermore, assume that the overlapping essentially cyclic update rule is used and the subproblem~\eqref{eq:BUpperBound} has a unique solution for every block~$\vartheta^r$. If $f(\cdot)$ is regular (with respect to the updated blocks) at every point in the set of stationary points~$\mathcal{X}^*$, then the iterates generated by the BSUM algorithm converges to the set of stationary points, i.e.,
\[
\lim_{r \rightarrow \infty} \quad d(x^r,\mathcal{X}^*) = 0.
\]
\end{itemize}
\end{coro}
\begin{proof}
The proof of both cases are similar to the proof of the BSUM algorithm with the simple cyclic update rule. Here we only present the proof for case (a). The proof of part (b) is similar.

Let $\{x^{r_j}\}$ be a convergent subsequence whose limit is denoted by~$z$. Consider every $T$ updating cycle along the subsequence $\{x^{r_j}\}$, namely, $\{(x^{r_j},x^{r_j +1},\ldots, x^{r_j+T-1})\}$. Since the number of different subblocks $\vartheta^r$ is finite, there must exist a (fixed) $T$ tuple of variable blocks, say $(\vartheta_0, \vartheta_1,\ldots, \vartheta_{T-1})$, that has been updated in infinitely many $T$ updating cycles. By restricting to the corresponding subsequence of $\{x^{r_j}\}$, we have
\[
x^{r_j+i+1}_{\vartheta_i} = \arg \min_{x_{\vartheta_i}} \; u_{\vartheta_i} (x_{\vartheta_i},x^{r_j + i }),\quad \forall\; i =0, 1,2,\ldots, T-1.
\]
The rest of the proof is the same as the proof of part (a) in Theorem~\ref{thm:ConvergenceBSUMConvex}. The only difference is that the steps of the proof need to be repeated for the blocks $(\vartheta_0, \vartheta_1,\ldots, \vartheta_{T-1})$ instead of $(1,\ldots,n)$.
%First we can prove that $\{x^{r_j +1}\} \rightarrow z$ using contradiction. By  writing the first order optimality condition for $i=0$ and letting $j \rightarrow \infty$, we obtain \emph{[[[I have trouble understanding this!!]]]}
%\[
%u'_{\vartheta_0}(x_{\vartheta_0},z;d_{\vartheta_0})\bigg|_{x_{t_1} = z_{t_1}} \geq 0, \quad \forall\; d_{\vartheta_0}.
%\]
%Repeating the above argument for other blocks and using \eqref{B3} will complete the proof.
\end{proof}

In the proof of Corollary~\ref{coro:ConvergenceBSUMConvex}, we first restrict ourselves to a fixed set of $T$ variable blocks that have been updated in infinitely many consecutive $T$ update cycles. Then, we use the same approach as in the proof of the convergence of cyclic update rule. Using the same technique, we can extend the results in Theorem~\ref{thm:ConvergenceBSCA} to the overlapping essentially cyclic update rule. More specifically, we have the following corollary.

\begin{coro}
Assume $f(\cdot)$ is smooth and the condition~\eqref{assptnNotubd} is satisfied. Furthermore, assume that $h(x,y)$ is strictly convex in $x$ and the overlapping essentially cyclic update rule is used in the BSCA algorithm. Then every limit point of the iterates generated by the BSCA algorithm is a stationary point of~\eqref{eq:OriginalProblem}.
\end{coro}

Notice that the overlapping essentially cyclic rule
is not applicable to the MISUM algorithm in which the update order of
the variables is given by the amount of improvement. However, one
can simply check that the proof of
Theorem~\ref{thm:ConvergenceMISUM} still applies to the case when
the blocks are allowed to have overlaps.

\section{Applications}
\label{sec:Applications}
In this section, we provide several applications of the algorithms
proposed in the previous sections.

\subsection{Linear Transceiver Design in Cellular
Networks}\label{subCellular}

Consider a $K$-cell wireless network where each base station~$k$ serves a set
$\mathcal{I}_k$ of users (see Fig. \ref{figIBC} for an
illustration). Let $i_k$ denote the $i$-th receiver in cell~$k$. For
simplicity, suppose that the users and the base stations are all
equipped with $N$ antennas. Let us define the set of all users as
$\mathcal{I} =\{i_k \mid 1\le k \le K, \; i \in \mathcal{I}_k\}.$
Let $d_{i_k}$ denote the number of data symbols transmitted
simultaneously to user $i_k$.
\begin{figure}[ht!]
 \centering
\includegraphics[width=0.4\linewidth]{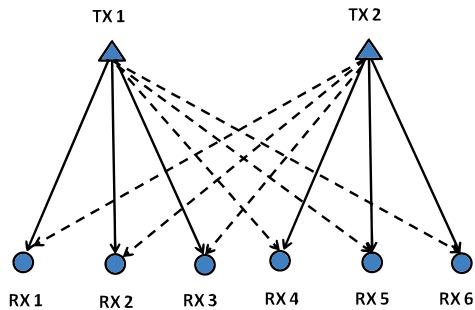}
\caption{\small The cellular network
model considered in Section \ref{subCellular}. The solid lines
represent the direct channels, while the dotted lines represent the
interfering channels.}\label{figIBC}
\end{figure}

When linear transceivers are used at the base stations and the
users, user~$i_k$'s received signal vector, denoted as
$\by_{i_k}\in\mathbb{C}^N$, can be written as
\begin{align}
\by_{i_k} &=
\underbrace{\bH_{i_kk}\bV_{i_k}\bs_{i_k}}_{\textrm{desired
signal}}+\underbrace{\sum_{\ell \neq i,
\ell=1}^{I_k}\bH_{i_kk}\bV_{\ell_k}\bs_{\ell_k}}_{\textrm{intracell
interference}}+\underbrace{\sum_{j\neq k, j=1}^K \sum_{\ell =
1}^{I_j}
\bH_{i_kj}\bV_{\ell_j}\bs_{\ell_j}+\bn_{i_k}}_{\textrm{intercell
interference plus noise}}, \;\forall\; i_k \in \mathcal{I}, \nonumber
\end{align}
where $\bV_{i_k}\in\mathbb{C}^{M\times d_{i_k}}$ is the linear
transmit beamformer used by base station $k$ for user $i_k$;
$\bs_{i_k}\in\mathbb{C}^{d_{i_k}\times1}$ is user $i_k$'s data
signal. The matrix $\bH_{i_kj}$ represents the channel from
transmitter~$j$ to receiver~$i_k$, and $\bn_{i_k}$ denotes the
complex additive white Gaussian noise with distribution
$\mathcal{CN}(0,\sigma_{i_k}^2 \bI)$. User~$i_k$ estimates the
intended message using a linear beamformer ~$\bU_{i_k}\in
\mathbb{C}^{M \times d_{i_k}}$: $\hat{\bs}_{i_k} = \bU^H_{i_k}
\by_{i_k}.$

%where $\bs_{i_k} \in \mathbb{C}^{d_{i_k} \times 1}$ is the data
%vector of user~$i_k$ with a normalized power $E[\bs_{i_k}
%\bs_{i_k}^H] = \bI$.
Treating interference as noise, the rate of user~$i_k$ is given
by{\small
\begin{align}
R_{i_k}(\bV)&= \log \det \bigg(\bI +  \bH_{i_k k}\bV_{i_k}
(\bV_{i_k})^H \bH_{i_k k}^H  \bigg( \sigma_{i_k}^2 \bI +
\sum_{(j,\ell) \neq (k,i)} \bH_{i_k j}\bV_{\ell_j} (\bV_{\ell_j})^H
\bH_{i_k j}^H \bigg)^{-1} \bigg). \nonumber
\end{align}}
We are interested in finding the beamformers $\bV$ such that the sum
of the users' rates are optimized
\begin{equation}
\label{eq:sum_utility_MIMO_IA}
\begin{split}
\max_{\{\bV_{i_k}\}_{i_k\in \mathcal{I}}} \quad &\sum_{k=1}^K\sum_{i=1}^{I_k} R_{i_k}(\bV)\\
\textrm{s.t.}\quad &\sum_{i=1}^{I_k}\tr(\bV_{i_k} \bV_{i_k}^H)\leq
\bar{P}_k, \;\; \forall\; k \in \mathcal{K}.
\end{split}
\end{equation}
Note that we have included a transmit power constraint for each base
station. It has been shown in \cite{luo08a} that solving \eqref{eq:sum_utility_MIMO_IA} is NP-hard. Therefore, we try to obtain the stationary solution for this problem. Furthermore,  we can no longer
straightforwardly apply the BSUM algorithm that updates
$\bV_{i_k}$'s cyclically. This is due to the fact that the users in
the set $\mathcal{I}_k$ share a common power constraint. Thus the
requirement for the separability of the constraints for different
block components in \eqref{eq:OriginalProblemBlock} is not
satisfied.

To devise an efficient and low complexity algorithm
for problem \eqref{eq:sum_utility_MIMO_IA}, we will
first transform this problem to a more suitable form. We first
introduce the function $f_{i_k} (\bU_{i_k}, \bV) \triangleq \log
\det \left(\bE_{i_k}^{-1}\right)$, where $\bE_{i_k}$ is the mean
square error (MSE) matrix given as{\small
\begin{equation}\nonumber
\begin{split}
\bE_{i_k} \triangleq(\bI-\bU_{i_k}^H \bH_{i_kk} \bV_{i_k})(\bI-
\bU_{i_k}^H \bH_{i_kk}\bV_{i_k})^H + \sum_{(\ell,j)\neq (i,k)}
\bU_{i_k}^H \bH_{i_kj} \bV_{\ell_j}\bV_{\ell_j}^H \bH_{i_kj}^H
\bU_{i_k}+\sigma_{i_k}^2\bU_{i_k}^H \bU_{i_k}.
\end{split}
\end{equation}}
In the subsequent presentation we will occasionally use the notation
$\bE_{i_k}(\bU_{i_k}, \bV)$ to make the dependency of the MSE matrix
and the transceivers explicit.

Taking the derivative of~$f_{i_k}(\bU_{i_k}, \bV)$ with respect
to~$\bU_{i_k}$ and checking the first order optimality condition, we
have{\small
\begin{align}
\arg \max_{\bU_{i_k}}f_{i_k} (\bU_{i_k}, \bV) = \bigg(
\sigma_{i_k}^2 \bI  + \sum_{(j,\ell)} \bH_{i_k j}\bV_{\ell_j}
\bV_{\ell_j}^H \bH_{i_k j}^H \bigg)^{-1} \bH_{i_k k}.
\bV_{i_k}\nonumber%\triangleq\bJ^{-1}_{i_k}\bH_{i_k k} \bV_{i_k}. \nonumber
\end{align}}
Plugging in the optimal value of~$\bU_{i_k}$ in~$f_{i_k}(\cdot)$, we
obtain $\max_{\bU_{i_k}}~ f_{i_k} (\bU_{i_k}, \bV) %&= %\log \det
%\bigg(\bI + \bH_{i_k k}\bV_{i_k} (\bV_{i_k})^H \bH_{i_k k}^H  \big(
%\sigma_{i_k}^2 \bI +  \sum_{(j,\ell) \neq (k,i)} \bH_{i_k
%j}\bV_{\ell_j} (\bV_{\ell_j})^H \bH_{i_k j}^H \big)^{-1} \bigg)
%\label{eq:FR}\nonumber\\
=R_{i_k}(\bV)\nonumber.$
Thus, we can rewrite the optimization
problem equivalently \eqref{eq:sum_utility_MIMO_IA} as\footnote{Such
equivalence is in the sense of one-to-one correspondence of both local and global
optimal solutions. See \cite{WMMSETSP} for a detailed argument. }
\begin{equation}
\label{eq:sum_utility__Multi1}
\begin{split}
\min_{\bV,\bU} \quad &\sum_{k=1}^K \sum_{i=1}^{I_k} \log\det(\bE_{i_k}) \\
\st \quad & \sum_{i=1}^{I_k}\tr\,(\bV_{i_k}\bV_{i_k}^H)\leq P_k,
\,\, \forall\; k \in\mathcal{K}.
\end{split}
\end{equation}

Notice the fact that the function $\log\det(\cdot)$ is a {\it
concave} function on its argument (see, e.g., \cite{boyd04}), then
for any feasible $\bE_{i_k}$, $\widehat{\bE}_{i_k}$, we have
\begin{align}
\log\det(\bE_{i_k})&\le
\log\det(\widehat{\bE}_{i_k})+\tr[\triangledown_{\bE_{i_k}}(\log\det(\widehat{\bE}_{i_k}))(\bE_{i_k}-\widehat{\bE}_{i_k})]\nonumber\\
&=\log\det(\widehat{\bE}_{i_k})+\tr[\widehat{\bE}^{-1}_{i_k}(\bE_{i_k}-\widehat{\bE}_{i_k})]
\triangleq u_{i_k}(\bE_{i_k},
\widehat{\bE}_{i_k})\label{eq:EUpperBound}
\end{align}

Utilizing the above transformation and the upper bound, we can again
apply the BSUM algorithm. Let $\bV$ and $\bU$ be two block
variables. Define
\begin{align}
u_{\bv}\left(\bV,(\widehat{\bV},\widehat{\bU})\right)&\triangleq\sum_{k=1}^K
\sum_{i=1}^{I_k} u_{i_k}(\bE_{i_k}(\bV, \widehat{\bU}_{i_k}),
{\bE}_{i_k}(\widehat{\bV},
\widehat{\bU}_{i_k}))\nonumber\\
u_{\bu}\left(\bU,(\widehat{\bV},\widehat{\bU})\right)&\triangleq\sum_{k=1}^K
\sum_{i=1}^{I_k} u_{i_k}(\bE_{i_k}(\widehat{\bV}, \bU_{i_k}),
{\bE}(\widehat{\bV}, \widehat{\bU}_{i_k}))\nonumber
\end{align}
In iteration $2r+1$, the algorithm solves the following problem
\begin{equation}
\label{eq:solving_V}
\begin{split}\min_{\bV} &\quad
u_\bv\left(\bV,(\bV^{2r},\bU^{2r})\right)\\
 &\quad\sum_{i=1}^{I_k}\tr\,(\bV_{i_k}\bV_{i_k}^H)\leq
P_k, \,\, \forall\; k \in\mathcal{K}.
\end{split}
\end{equation}

In iteration $2r+2$, the algorithm solves the following
(unconstrained) problem
\begin{equation}
\label{eq:solving_U}
\begin{split}
\min_{\bU} \quad &u_{\bu}\left(\bU,(\bV^{2r+1},\bU^{2r})\right)
\end{split}
\end{equation}
The above BSUM algorithm for solving \eqref{eq:sum_utility_MIMO_IA} is called WMMSE algorithm in the reference \cite{WMMSETSP}.

Due to \eqref{eq:EUpperBound}, we must have that
\begin{align}
&u_\bv\left(\bV,(\bV^{2r},\bU^{2r})\right)\ge \sum_{k=1}^K \sum_{i=1}^{I_k}\log\det(\bE_{i_k}(\bV,\bU^{2r}_{i_k})),\quad {\rm for~all~feasible~}\bV,~\forall~i_k\nonumber\\
&u_{\bu}\left(\bU,(\bV^{2r+1},\bU^{2r})\right)\ge \sum_{k=1}^K
\sum_{i=1}^{I_k}\log\det(\bE_{i_k}(\bV^{2r+1},\bU_{i_k}))\quad {\rm
for~all~}\bU_{i_k},~\forall~i_k\nonumber
\end{align}
Moreover, other conditions in Assumption \ref{AssumptionB} are also
satisfied for $u_{\bv}(\cdot)$ and $u_{\bu}(\cdot)$. Thus the
convergence of the WMMSE algorithm to a stationary solution of problem
\eqref{eq:sum_utility__Multi1} follows directly from Theorem~\ref{thm:ConvergenceBSUMConvex}.

We briefly mention here that the main benefit of using the BSUM
approach for solving problem \eqref{eq:sum_utility__Multi1} is that
in each step, the problem \eqref{eq:solving_V} can be decomposed
into $K$ independent convex subproblems, one for each base station
$k\in\mathcal{K}$. Moreover, the solutions for these $K$ subproblems
can be simply obtained in closed form (subject to an efficient
bisection search). For more details on this algorithm, we refer the
readers to \cite{WMMSETSP} and \cite{GroupingRazaviyayn}.

The BSUM approach has been extensively used for resource allocation
in wireless networks, for example
\cite{ILAGinnakisKim,PowerAllocationGeometric,Shi:2009,Ng10,hong12_icassp},
and \cite{InterferencepricingShi}. However, the convergence of most
of the algorithms was not rigorously established.

\subsection{Proximal Minimization Algorithm}
The classical proximal minimization algorithm (see, e.g.,
\cite[Section 3.4.3]{Bertsekas_Book_Distr}) obtains a solution of
the problem $\min_{\bx\in\mathcal{X}}f(\bx)$ by solving an
equivalent problem
\begin{align}
\min_{\bx\in\mathcal{X},\by\in\mathcal{X}}f(\bx)+\frac{1}{2c}\|\bx-\by\|_2^2,\label{problemProximal}
\end{align}
where $f(\cdot)$ is a convex function, $\mathcal{X}$ is a closed
convex set, and $c>0$ is a scalar parameter. The equivalent problem
\eqref{problemProximal} is attractive in that it is strongly convex
in both $x$ and $y$ (but not jointly) so long as $f(x)$ is convex.
This problem can be solved by performing the following two steps in
an alternating fashion
\begin{align}
\bx^{r+1}&=\arg\min_{\bx\in\mathcal{X}}\left\{f(x)+\frac{1}{2c}\|\bx-\by^r\|^2_2\right\}\label{eq:ProximalComputationX}\\
\by^{r+1}&=\bx^{r+1}\label{eq:ProximalComputationY}.
\end{align}
Equivalently, let $u(\bx;\bx^r)\triangleq
f(\bx)+\frac{1}{2c}\|\bx-\bx^r\|_2^2$, then the iteration
\eqref{eq:ProximalComputationX}--\eqref{eq:ProximalComputationY} can
be written as
\begin{align}
\bx^{r+1}=\arg\min_{\bx\in\mathcal{X}}u(\bx,\bx^r).\label{eq:ProximalEquivalence}
\end{align}
It can be straightforwardly checked that for all $\bx,
\bx^r\in\mathcal{X}$, the function $u(\bx, \bx^r)$ serves as an
upper bound for the function $f(\bx)$. Moreover, the conditions
listed in Assumption \ref{AssumptionA} are all satisfied. Clearly,
the iteration \eqref{eq:ProximalEquivalence} corresponds to the SUM
algorithm discussed in Section \ref{sec:SUM}. Consequently, the
convergence of the proximal minimization procedure can be obtained
from Theorem \ref{thm:ConvergenceSUMConvex}.

The proximal minimization algorithm can be generalized in the
following way. Consider the problem
\begin{align}
\min_{\bx}&\quad f(\bx_1,\cdots,\bx_n)\label{problemProximalGeneral}\\
{\rm s.t.}&\quad \bx_i\in\mathcal{X}_i,~i=1,\cdots,n,\nonumber
\end{align}
where $\{\mathcal{X}_i\}_{i=1}^{n}$ are closed convex sets,
$f(\cdot)$ is convex in each of its block components, but not
necessarily strictly convex. A straightforward application of the
BCD procedure may fail to find a stationary solution for this
problem, as the per-block subproblems may contain multiple
solutions. Alternatively, we can consider an {\it alternating
proximal minimization} algorithm \cite{ProximalBCD}, in each
iteration of which the following subproblem is solved
\begin{align}
\min_{\bx_i}&\quad f(\bx^r_1,\ldots, \bx_{i-1}^r,\bx_i,\bx_{i+1}^r ,\ldots, \bx^r_n)+\frac{1}{2c}\|\bx_i-\bx^r_i\|^2_2\\
{\rm s.t.}&\quad \bx_i\in\mathcal{X}_i\nonumber.
\end{align}

It is not hard to see that this subproblem always admits a unique
solution, as the objective is a strictly convex function of
$\bx_i$. Let $u_i(\bx_i,\bx^r)\triangleq f(\bx^r_1,\cdots,\bx_i,
\cdots \bx^r_n)+\frac{1}{2c}\|\bx_i-\bx^r_i\|^2_2$. Again for each
$\bx_i\in\mathcal{X}_i$ and
$\bx^r\in\mathcal\prod_{j}\mathcal{X}_j$, the function
$u_i(\bx_i,\bx^r)$ is an upper bound of the original objective
$f(\bx)$. Moreover, all the conditions in Assumption
\ref{AssumptionB} are satisfied. Utilizing Theorem
\ref{thm:ConvergenceBSUMConvex}, we conclude that the alternating
proximal minimization algorithm must converge to a stationary
solution of the problem \eqref{problemProximalGeneral}. Moreover,
our result extends those in \cite{ProximalBCD} to the case of
nonsmooth objective function as well as the case with
iteration-dependent coefficient $c$. The latter case, which was also studied in the contemporary work \cite{WataoYinBCD}, will be demonstrated in an example for tensor decomposition shortly.

\subsection{Proximal Splitting Algorithm}
The proximal splitting algorithm
(see, e.g., \cite{Combettes09}) for nonsmooth optimization is also a
special case of the BSUM algorithm. Consider the following problem
\begin{align}
\min_{\bx\in\mathcal{X}}f_1(\bx)+f_2(\bx)\label{problemProximalSplitting}
\end{align}
where $\mathcal{X}$ is a closed and convex set. Furthermore, $f_1$
is convex and lower semicontinuous; $f_2$ is convex and has
Lipschitz continuous gradient, i.e., $\|\nabla f_2(\bx)-\nabla
f_2(\by)\|\le \beta \|\bx-\by\|$, $\forall~\bx,\by\in\mathcal{X}$
and for some $\beta>0$.

Define the proximity operator ${\rm
prox}_{f_i}:\mathcal{X}\to\mathcal{X}$ as
\begin{align}
{\rm
prox}_{f_i}(\bx)=\arg\min_{\by\in\mathcal{X}}f_i(\by)+\frac{1}{2}\|\bx-\by\|^2.
\end{align}
The following forward-backward splitting iteration can be used to
obtain a solution for problem \eqref{problemProximalSplitting}
\cite{Combettes09}:
\begin{align}
\bx^{r+1}=\underbrace{{\rm prox}_{\gamma f_1}}_{\rm backward~
step}\underbrace{(\bx^r-\gamma\nabla f_2(\bx^r))}_{\rm
forward~step}\label{eq:forwardbackward}
\end{align}
where $\gamma\in[\epsilon,2/\beta-\epsilon]$ with $\epsilon\in
]0,\min\{1,1/\beta\}[$. Define
\begin{align}
u(\bx,\bx^r)\triangleq
f_1(\bx)+\frac{1}{2\gamma}\|\bx-\bx^r\|^2+\langle\bx-\bx^r,\nabla
f_2(\bx^r)\rangle+f_2(\bx^r).
\end{align}
We first show that the iteration \eqref{eq:forwardbackward} is
equivalent to the following iteration
\begin{align}
\bx^{r+1}=\arg\min_{\bx\in\mathcal{X}}u(\bx,\bx^r).
\end{align}
From the definition of the prox operation, we have
\begin{align}
{\rm prox}_{\gamma f_1}(\bx^r-\gamma\nabla
f_2(\bx^r))&=\arg\min_{\bx\in\mathcal{X}}\gamma
f_1(\bx)+\frac{1}{2}\|\bx-\bx^r+\gamma\nabla
f_2(\bx^r)\|^2_2\nonumber\\
&=\arg\min_{\bx\in\mathcal{X}}f_1(\bx)+\frac{1}{2\gamma
}\|\bx-\bx^r\|_2^2+\langle\bx-\bx^r,\nabla f_2(\bx^r)\rangle\nonumber\\
&=\arg\min_{\bx\in\mathcal{X}}u(\bx,\bx^r).\nonumber
\end{align}

We then show that $u(\bx, \bx^r)$ is an upper bound of the original
function $f_1(\bx)+f_2(\bx)$, for all $\bx,\bx^r\in\mathcal{X}$.
Note that from the well known Descent Lemma \cite[Proposition
A.32]{Bertsekas_Book_Nonlinear}, we have that
\begin{align}
f_2(\bx)&\le
f_2(\bx^r)+\frac{\beta}{2}\|\bx-\bx^r\|^2+\langle\bx-\bx^r,\nabla
f_2(\bx^r)\rangle\nonumber\\
&\le
f_2(\bx^r)+\frac{1}{2\gamma}\|\bx-\bx^r\|^2+\langle\bx-\bx^r,\nabla
f_2(\bx^r)\rangle\nonumber
\end{align}
where the second inequality is from the definition of $\gamma$.
This result implies that $u(\bx,\by)\ge
f_1(\bx)+f_2(\bx),~\forall~\bx,\by\in\mathcal{X}$. Moreover, we can
again verify that all the other conditions in Assumption
\ref{AssumptionA} is true. Consequently, we conclude that the
forward-backward splitting algorithm is a special case of the SUM
algorithm.

Similar to the previous example, we can generalize the
forward-backward splitting algorithm to the problem with multiple
block components. Consider the following problem
\begin{align}
\min&\quad
\sum_{i=1}^{n}f_i(\bx_i)+f_{n+1}(\bx_1,\cdots,\bx_n)\label{problemProximalSplittingGeneral}\\
{\rm s.t.}&\quad\bx_i\in\mathcal{X}_i,i=1,\cdots,n\nonumber
\end{align}
where $\{\mathcal{X}_i\}_{i=1}^{n}$ are a closed and convex sets.
Each function $f_i(\cdot)$, $i=1,\cdots n$ is convex and lower
semicontinuous w.r.t. $\bx_i$; $f_{n+1}(\cdot)$ is convex and has
Lipschitz continuous gradient w.r.t. each of the component $\bx_i$,
i.e., $\|\nabla f_{n+1}(\bx)-\nabla f_{n+1}(\by)\|\le \beta_i
\|\bx_i-\by_i\|$, $\forall~\bx_i,\by_i\in\mathcal{X}_i,
i=1,\cdots,n$. Then the following block forward-backward splitting
algorithm can be shown as a special case of the BSUM algorithm, and
consequently converges to a stationary solution of the problem
\eqref{problemProximalSplittingGeneral}
\begin{align}
\bx_i^{r+1}={{\rm prox}_{\gamma
f_i}}{(\bx_i^r-\gamma\nabla_{\bx_i} f_{n+1}(\bx^r))},\quad i=1,2,...,n,\nonumber
\end{align}
where $\gamma\in[\epsilon_i,2/\beta_i-\epsilon_i]$ with
$\epsilon_i\in ]0,\min\{1,1/\beta_i\}[$.

\subsection{CANDECOMP/PARAFAC Decomposition of Tensors}
Another application of the proposed method is in CANDECOMP/PARAFAC (CP) decomposition of tensors. Given a tensor $\mathfrak{X} \in \re^{m_1 \times m_2 \times \ldots \times m_n}$ of order $n$, the idea of CP decomposition is to write the tensor as the sum of rank-one tensors:
\[
\mathfrak{X} = \sum_{r=1}^R \mathfrak{X}_r,
\]
where $\mathfrak{X}_r = a_{1r} \circ a_{2r} \circ \ldots \circ a_{nr}$ and $a_{ir} \in \mathbb{R}^{m_i}$. Here the notation $``\circ"$ denotes the outer product.\\

In general, finding the CP decomposition of a given tensor is
NP-hard \cite{HastadTensorRankNPhard}. In practice, one of the most
widely accepted algorithms for computing the CP decomposition of a
tensor is the Alternating Least Squares (ALS) algorithm
\cite{TensorReviewKolda,CPDecompositionFaber,CPDecompositionTomasi}.
The ALS algorithm proposed in
\cite{CPFirstWorkCarroll,CPFirstWorkHarshman} is in essence a BCD
method. For ease of presentation, we will present the ALS
algorithm only for tensors of order three.\\

Let $\mathfrak{X} \in \mathbb{R}^{I\times J\times K}$ be a third
order tensor. Let $(A;B;C)$ represent the following decomposition
\[
(A;B;C) \triangleq \sum_{r=1}^R a_r \circ b_r \circ c_r,
\]
where $a_r$ (resp. $b_r$ and $c_r$) is the $r$-th column of $A$
(resp. $B$ and $C$). The ALS algorithm minimizes the difference
between the original and the reconstructed tensors
\begin{equation}\label{eq:ALSopt}
\min_{A,B,C} \quad \|\mathfrak{X} - (A;B;C)\|,
\end{equation}
where $A\in \re^{I\times R}$, $B \in \re^{J \times R}$, $C\in
\re^{K\times R}$, and $R$ is the rank of the tensor.\\

The ALS approach is a special case of the BCD algorithm in which the
three blocks of variables $A,B,$ and $C$ are cyclically updated. In
each step of the computation when two blocks of variables are held
fixed, the subproblem becomes the quadratic least squares
problem and admits closed form updates (see \cite{TensorReviewKolda}). \\

One of the well-known drawbacks of the ALS algorithm is the {\it
swamp} effect where the objective value remains almost constant
for many iterations before starting to decrease again. Navasca
{\it et al.} in \cite{CPSwampTikhonov} observed that adding a
proximal term in the algorithm could help reducing the swamp effect.
More specifically, at each iteration $r$ the algorithm proposed in
\cite{CPSwampTikhonov} solves the following problem for updating the
variables:
\begin{equation}\label{EQ:CPTikhonov}
\|\mathfrak{X}- (A;B;C)\|^2 + \lambda \|A-A^r\|^2 + \lambda \|B-B^r\|^2 + \lambda \|C-C^r\|^2,
\end{equation}
where $\lambda \in \re$ is a positive constant. As discussed before,
this proximal term has been considered in different optimization contexts
and its convergence has been already showed in
\cite{ProximalBCD}. An interesting numerical observation in \cite{CPSwampTikhonov} is that
decreasing the value of $\lambda$ during the algorithm can noticeably improve
the convergence of the algorithm. Such iterative decrease of $\lambda$ can
be accomplished in a number of different ways. Our numerical experiments show
that the following simple approach to update $\lambda$ can significantly improve
the convergence of the ALS algorithm and substantially reduce the swamp effect:
\begin{equation}\label{EQ:updaterulLambda}
\lambda^r = \lambda_0 + \lambda_1\frac{\|\mathfrak{X} - (A^r;B^r;C^r)\|}{\|\mathfrak{X}\|},
\end{equation}
where $\lambda^r$ is the proximal coefficient $\lambda$ at iteration~$r$. Theorem~\ref{thm:ConvergenceBSUMConvex} implies the convergence is guaranteed even with this update rule of $\lambda$, whereas the convergence result of \cite{ProximalBCD} does not apply in this case since the proximal coefficient is changing during the iterations.\\

Figure~\ref{FIG:TensorALS} shows the performance of different algorithms for the example given in \cite{CPSwampTikhonov} where the tensor~$\mathfrak{X}$ is obtained from the decomposition
\begin{equation} \nonumber
A = \left[
\begin{array}{ccc}
1 & \cos\theta & 0\\
0 & \sin\theta & 1 \\
\end{array}
\right],\quad
B = \left[
\begin{array}{ccc}
3 & \sqrt{2}\cos\theta & 0\\
0 & \sin\theta & 1 \\
0 & \sin\theta & 0 \\
\end{array}
\right],\quad
C = \left[
\begin{array}{ccc}
1 & 0 & 0\\
0 & 1 & 0 \\
0 & 0 & 1 \\
\end{array}
\right].
\end{equation}
The vertical axis is the value of the objective function where the horizontal axis is the iteration number.
In this plot, {\it ALS} is the classical alternating least squares algorithm. The curve for {\it Constant Proximal} shows the performance of the BSUM algorithm when we use the objective function in \eqref{EQ:CPTikhonov} with $\lambda= 0.1$. The curve for {\it Diminishing Proximal} shows the performance of block coordinate descent method on \eqref{EQ:CPTikhonov} where the weight $\lambda$ decreases iteratively according to \eqref{EQ:updaterulLambda} with $\lambda_0 = 10^{-7}, \lambda_1 = 0.1$. The other two curves {\it MBI} and {\it MISUM} correspond to the maximum block improvement algorithm and the MISUM algorithm. In the implementation of the MISUM algorithm, the proximal term is of the form in \eqref{EQ:CPTikhonov} and the weight $\lambda$ is updated based on \eqref{EQ:updaterulLambda}. \\
\begin{figure}[ht!]
 \centering
\includegraphics[width=4in]{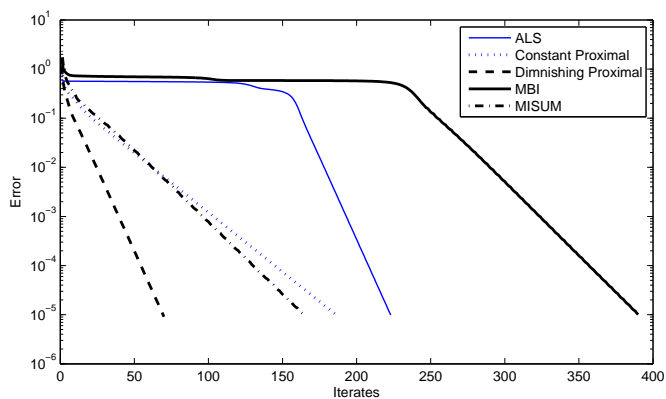}
\caption{Convergence of Different Algorithms}\label{FIG:TensorALS}
\end{figure}

Table~\ref{Tab:AverageALSTensor} represents the average number of iterations required to get an objective value less than $\epsilon = 10^{-5}$ for different algorithms. The average is taken over 1000 Monte-Carlo runs over different initializations. The initial points are generated randomly where the components of the variables $A,B,$ and $C$ are drawn independently from the uniform distribution over the unit interval $[0, 1]$. As it can be seen, adding a diminishing proximal term significantly improves the convergence speed of the ALS algorithm.

\begin{table}
\centering
\begin{tabular}{|c|c|}
  \hline
  {\rm Algorithm} & {\rm Average number of iterations for convergence}\\
  \hline
  {\rm ALS} & 277\\
  \hline
  {\rm Constant Proximal} & 140\\
  \hline
  {\rm Diminishing Proximal} & 78 \\
  \hline
  {\rm MBI} & 572 \\
  \hline
  {\rm MISUM} & 175 \\
  \hline
\end{tabular}
\caption{Average number of iterations for convergence}
\vspace{-0.5cm} \label{Tab:AverageALSTensor}
\end{table}

\subsection{Expectation Maximization Algorithm}
The expectation maximization algorithm (EM) in \cite{EMDempster} is an iterative procedure for maximum likelihood estimation when some of the random variables are unobserved/hidden. Let $w$ be the observed random vector which is used for estimating the value of $\theta$. The maximum likelihood estimate of $\theta$ can be given as
\begin{equation}\label{EQ:EMML}
\hat{\theta}_{\rm ML} = \arg \max_{\theta} \; \ln p(w|\theta).
\end{equation}
Let the random vector $z$ be the hidden/unobserved variable. The EM algorithm starts from an initial estimate $\theta^0$ and generates a sequence $\{\theta^r\}$ by repeating the following steps:
\begin{itemize}
\item E-Step: Calculate $g(\theta,\theta^r) \triangleq \mathbb{E}_{z|w,\theta^r} \{\ln p(w,z|\theta)\}$
\item M-Step: $\theta^{r+1} = \arg \max_\theta \; g(\theta,\theta^r)$
\end{itemize}
The EM-algorithm can be viewed as a special case of SUM algorithm \cite{EMTutorial}. In fact, we are interested in solving the following optimization problem
\[
\min_\theta \quad -\ln p(w|\theta).
\]
The objective function could be written as
\begin{align}
-\ln p(w|\theta) & = -\ln \;\;\mathbb{E}_{z|\theta}\; p(w|z,\theta)\nonumber\\
& = -\ln \;\;\mathbb{E}_{z|\theta} \left[\frac{p(z|w,\theta^r) p(w|z,\theta)}{p(z|w,\theta^r)}\right]\nonumber\\
& = -\ln \;\;\mathbb{E}_{z|w,\theta^r} \left[\frac{p(z|\theta) p(w|z,\theta)}{p(z|w,\theta^r)}\right]\nonumber\\
& \leq - \mathbb{E}_{z|w,\theta^r} \ln\left[\frac{p(z|\theta)p(w|z,\theta)}{p(z|w,\theta^r)}\right]\nonumber\\
& = - \mathbb{E}_{z|w,\theta^r} \ln p(w,z|\theta) + \mathbb{E}_{z|w,\theta^r} \ln p(z|w,\theta^r)\nonumber\\
& \triangleq u(\theta,\theta^r),\nonumber
\end{align}
where the inequality is due to the Jensen's inequality and the third equality follows from a simple change of the order of integration for the expectation. Since $\mathbb{E}_{z|w,\theta^r} \ln p(z|w,\theta^r)$ is not a function of $\theta$, the M-step in the EM-algorithm can be written as
\[
\theta^{r+1} = \arg\max_{\theta} u(\theta,\theta^r).
\]
Furthermore, it is not hard to see that $u(\theta^r,\theta^r) = - \ln p(w|\theta^r)$. Therefore, under the smoothness assumption, Proposition~\ref{lemma:SUMAssumption} implies that Assumption~\ref{AssumptionA} is satisfied. As an immediate consequence, the EM-algorithm is a special case of the SUM algorithm. Therefore, our result implies not only the convergence of the EM-algorithm, but also the convergence of the EM-algorithm with Gauss-Seidel/coordinatewise update rule (under the assumptions of Theorem~\ref{thm:ConvergenceBSUMConvex}). In fact in the block coordinate EM-algorithm (BEM), at each M-step, only one block is updated. More specifically, let $\theta = (\theta_1,\ldots, \theta_n)$ be the unknown parameter. Assume $w$ is the observed vector and $z$ is the hidden/unobserved variable as before. The BEM algorithm starts from an initial point $\theta^0 = (\theta^0_1,\ldots,\theta^0_n)$ and generates a sequence $\{\theta^r\}$ according to the algorithm in Figure~\ref{Fig:BEMAlgorithm}.

\begin{figure*}[t]
\centering
\begin{tabular}{|p{4.7in}|}
%\hline \textbf{input: $x_0,\alpha_k$, output: $x_k$}\\
\hline
\begin{itemize}
\item [1] \;Initialize with $\theta^0$ and set $r = 0$
\item [2] \; \textbf{repeat}
\item [3] \quad $r = r+1$, $i = r \;{\rm mod}\; n +1$
\item [4]  \quad E-Step: $g_i(\theta_i,\theta^r) = \mathbb{E}_{z|w,\theta^r} \{\ln p(w,z|\theta_1^r,\ldots,\theta_{i-1}^r,\theta_i, \theta_{i+1}^{r},\ldots,\theta_n^{r})\}$
\item [5]  \quad M-Step: $\theta_{i}^{r+1} = \arg \max_{\theta_i} \; g_i(\theta_i,\theta^r)$
\item [6] \; \textbf{until} some convergence criterion is met
\end{itemize}
\\
\hline
\end{tabular}\vspace{1.2em}
\caption{Pseudo code of the BEM algorithm} \label{Fig:BEMAlgorithm}
\end{figure*}
The motivation behind using the BEM algorithm instead of the EM algorithm could be
the difficulties in solving the M-step of EM for the entire set of variables, while solving the same problem per block of variables is easy. To the best of our knowledge, the BEM algorithm and its convergence behavior have not been analyzed before.

\subsection{Concave-Convex Procedure/Difference of  Convex Functions}
A popular algorithm for solving unconstrained problems, which also belongs to the class of successive upper-bound minimization, is the Concave-Convex Procedure (CCCP) introduced in \cite{CCCP}. In CCCP, also known as the difference of convex functions (DC) programming, we consider the unconstrained problem
\[
\min_{x\in \mathbb{R}^m} \; f(x),
\]
where $f(x) = f_{cve}(x) + f_{cvx}(x), \forall\;  x\in \re^m$; where $f_{cve}(\cdot)$ is a concave function and $f_{cvx}(\cdot)$ is convex. The CCCP generates a sequence $\{x^r\}$ by solving the following equation:
\begin{equation} \nonumber
\nabla f_{cvx} (x^{r+1}) = -\nabla f_{cve} (x^r),
\end{equation}
which is equivalent to
\begin{equation} \label{EQ:CCCP}
x^{r+1} = \arg \min_{x} \; g(x,x^r),
\end{equation}
where $g(x,x^r) \triangleq f_{cvx}(x) + (x - x^r)^T \nabla f_{cve}(x^r) + f_{cve}(x^r)$. Clearly, $g(x,x^r)$ is a tight convex upper-bound of $f(x)$ and hence CCCP is a special case of the SUM algorithm and its convergence is guaranteed by Theorem~\ref{thm:ConvergenceSUMConvex} under certain assumptions. Furthermore, if the updates are done in a block coordinate manner, the algorithm becomes a special case of BSUM whose convergence is guaranteed by Theorem~\ref{thm:ConvergenceBSUMConvex}. To the best of our knowledge, the block coordinate version of CCCP  algorithm and its convergence has not been studied before.

%A more detailed desTable \ref{Fig:Splitting} can be used to obtain a
%solution for problem \eqref{problemProximalSplittingGeneral}
%\begin{figure*}[t]
%\centering
%\begin{tabular}{|p{3.5in}|}
%%\hline \textbf{input: $x_0,\alpha_k$, output: $x_k$}\\
%\hline
%\begin{itemize}
%\item [1] Fix $\epsilon\in ]0,\min\{1,1/\beta\}[$; Choose
%$\bx^0\in\mathcal{X}$; Set $r=0$
%\item [2] {\bf repeat}
%\item [3] \quad Let $r=r+1$
%\item [4] \quad Choose $\gamma\in[\epsilon,2/\beta-\epsilon]$
%\item [5] \quad Set $\by^r=\bx^r-\gamma\nabla f_2(\bx^r)$
%\item [6] \quad Set $\bx^{r+1}=\bx^r+({\rm prox}_{\gamma
%f_1}\by^r-\bx^r)$
%\item [7]{\bf until} some convergence criterion is met
%\end{itemize}
%\\
%\hline
%\end{tabular}\vspace{1.2em}
%\caption{The Forward-Backward Proximal Splitting Algorithm
%\cite{Combettes09}} \label{Fig:Splitting}
%\end{figure*}

\bibliographystyle{IEEEtranS}
\bibliography{IEEEabrv,biblio}

\end{document}